\documentclass[preprint,3p]{article}

\usepackage{style}
\usepackage{authblk}
\usepackage[margin=1in]{geometry}

\title{Bayesian~model~calibration with interpolating~polynomials based on adaptively~weighted~Leja~nodes}

\author[1,2]{L.M.M.~van~den~Bos\footnote{Corresponding author: \texttt{l.m.m.van.den.bos@cwi.nl}}}
\author[1]{B.~Sanderse}
\author[2]{W.A.A.M.~Bierbooms}
\author[2]{G.J.W.~van~Bussel}

\affil[1]{Centrum~Wiskunde~\&~Informatica, P.O.~Box 94079, 1090~GB, Amsterdam}
\affil[2]{Delft~University~of~Technology, P.O.~Box 5, 2600~AA, Delft}

\begin{document}
\maketitle

\begin{abstract}
\noindent An efficient algorithm is proposed for Bayesian model calibration, which is commonly used to estimate the model parameters of non-linear, computationally expensive models using measurement data. The approach is based on Bayesian statistics: using a prior distribution and a likelihood, the posterior distribution is obtained through application of Bayes' law. Our novel algorithm to accurately determine this posterior requires significantly fewer discrete model evaluations than traditional Monte Carlo methods. The key idea is to replace the expensive model by an interpolating surrogate model and to construct the interpolating nodal set maximizing the accuracy of the posterior. To determine such a nodal set an extension to weighted Leja nodes is introduced, based on a new weighting function. We prove that the convergence of the posterior has the same rate as the convergence of the model. If the convergence of the posterior is measured in the Kullback--Leibler divergence, the rate doubles. The algorithm and its theoretical properties are verified in three different test cases: analytical cases that confirm the correctness of the theoretical findings, Burgers' equation to show its applicability in implicit problems, and finally the calibration of the closure parameters of a turbulence model to show the effectiveness for computationally expensive problems.
\end{abstract}
\textbf{Keywords:} Bayesian model calibration, Interpolation, Leja nodes, Surrogate modeling

\section{Introduction}
Estimating model parameters from measurements is a problem of frequent occurrence in many fields of engineering and many different approaches exist to solve this problem. We consider non-linear calibration problems (or inverse problems) where a forward evaluation of the model is computationally expensive. The approach we follow is of a stochastic nature: the unknown parameters are modeled using probability distributions and information about these parameters is inferred using Bayesian statistics. This approach is often called \emph{Bayesian model calibration}.

Bayesian model calibration \cite{Kennedy2001,Ray2015,Sargsyan2015} is a systematic way to calibrate the parameters of a computational model. By means of a statistical model to describe the relation between the model and the data, the calibrated parameters are obtained in the form of a random variable (called the \emph{posterior}) by means of Bayes' law. These random variables can then be used to assess the uncertainty in the model and to make future predictions. This procedure is well-known in the field of Bayesian statistics, where the goal is to infer unmeasured quantities from data. The calibration approach has already been applied many times, for example to calibrate the closure parameters of turbulence models \cite{Edeling2014-2,Cheung2011}. A similar example is considered in this work.

Possibly the largest drawback of Bayesian model calibration is the expensive sampling procedure that is necessary. Because the posterior depends to a large extent on the model, which is only known implicitly (e.g.\ a computer code numerically solving a partial differential equation), determining a sample from the posterior is mostly done using Markov chain Monte Carlo (MCMC) methods \cite{Metropolis1953,Hastings1970}, requiring many expensive model evaluations. Improvements have been made to accelerate these MCMC methods, e.g.\ the DREAM algorithm \cite{Vrugt2009} or adaptive sampling \cite{Yeung2002}. Replacing the sampling procedure itself is also possible, e.g.\ methods based on sparse grids \cite{Chauveau2002,Ma2009} or Approximate Bayesian Computation \cite{Beaumont2002,Csillery2010,Lopes2009}. However, this encompasses stringent assumptions on the statistical model or still requires many model runs as the shape of the posterior is unknown.

A different approach is followed in the current article. In essence we are following the approach of \citet{Marzouk2007}, which has been used several times in literature \cite{Marzouk2009,Birolleau2014,Zhang2013,Zeng2012,Absi2014,Prudhomme2015,Oliver2011,Yan2019}. The key idea in our procedure is to replace the model in the calibration step with a \emph{surrogate} (or \emph{response surface}) that approximates the computationally expensive model. MCMC can then be used to sample the resulting posterior without a large computational overhead.

Various approaches to construct this surrogate in a Bayesian context exist, for example Gaussian process emulators \cite{Stuart2017} or non-intrusive polynomial approximations \cite{Xiu2010}. In this work the latter is considered, because polynomial approximations provide high order (up to exponential) convergence for sufficiently smooth functions. Contrary to the commonly used pseudo-spectral projection methods, which are commonly known as generalized Polynomial Chaos Expansions, we choose to use interpolation of the computationally expensive model. The reason for this is that the error of a polynomial interpolant is usually measured using the absolute error (the $L^\infty$ norm), contrary to the mean squared error (the $L^2$ norm) that is used for the pseudo-spectral approaches. As the model is used as input in the Bayesian analysis, having absolute error bounds on the surrogate significantly simplifies the analysis. Moreover, the convergence of a pseudo-spectral expansion deteriorates significantly if the surrogate is not constructed using the statistical model \cite{Lu2015}. This happens in particular if the expansion is constructed with respect to the prior (which is the usual approach) and the likelihood contains significant information (i.e.\ their relative entropy is high).

The interpolating surrogate model is built using Leja nodes. Probability density functions can be incorporated using weighted Leja nodes \cite{Jantsch2016,Narayan2014}. We extend weighted Leja nodes to \emph{adaptively} refine the interpolating polynomial by using obtained posterior information. As extensive theory about interpolation polynomials exists (e.g.\ \cite{Ibrahimoglu2016}), we can prove convergence of the estimated posterior with mild assumptions on the likelihood. This extends previous work \cite{Marzouk2007,Birolleau2014}, in which the likelihood is assumed to be Gaussian. The end result is an interpolating polynomial that can be used in conjunction with the likelihood and the prior to obtain statistics of the posterior.

To demonstrate the applicability of our methodology, we will employ three different classes of test problems. The first class consist of functions that are known explicitly and can be evaluated fast and accurately. We will use these to show the effectiveness of our nodal set compared to commonly used methods. The second class consists of problems that are defined implicitly, but do not require significant computational power to solve. For this, we employ the one-dimensional Burgers' equation. In this case, it is possible to compare the estimated posterior with a posterior determined using Monte Carlo methods. The last class consists of problems of such large complexity that a quantitative comparison with a true posterior is not possible anymore. As example we consider the calibration of closure coefficients of the Spalart--Allmaras turbulence model.

This paper is set up as follows. First, we discuss Bayesian model calibration and introduce the adaptively weighted Leja nodes. In Section~\ref{sec:convergence} the theoretical properties of the algorithm are studied and its convergence is assessed. Section~\ref{sec:numerics} contains numerical tests that show evidence of the theoretical findings and in Section~\ref{sec:conclusion} conclusions are drawn.

\section{Bayesian model calibration with a surrogate}
\label{sec:bayes}
The focus is on the stochastic calibration of computationally expensive (possibly implicitly defined) models. We denote this model by $u: \Omega \rightarrow \mathbb{R}$, with $\Omega \subset \mathbb{R}^d$ ($d = 1, 2, \dots$). In this work we will not focus on the specific construction of this model, but for example $u$ arises as a solution of a set of partial differential equations. Without loss of generality, we assume that $u$ is a scalar quantity. $u$ depends on $d$ parameters, which we will denote as vector $\theb = \trans{(\vartheta_1, \dots, \vartheta_d)} \in \Omega$. One can think of $\theb$ as parameters inherent to the model, such as fitting parameters or other closure parameters.

The goal of Bayesian model calibration is to infer knowledge about the model parameters, given measurement data of the process modeled by $u$. To this end, we assume that a vector of measurements $\mathbf{z} = \trans{(z_1, \dots, z_n)}$ is given, with $z_k \in \mathbb{R}$. This data can be provided by various means, for example by measurements or by the results of a high-fidelity model. Using parameters $\theb$, a statistical model is formulated describing a relation between the model $u(\theb)$ and the data $\mathbf{z}$ by means of random variables that model among others discrepancy, error, and uncertainty. For example, these random variables account for measurement errors and numerical tolerances. Using Bayesian statistics \cite{Gelman2013}, the posterior of the parameters is formulated by means of a probability density function (PDF).

Throughout this article we let $p(\theb)$ be the prior, a PDF containing all prior information of $\theb$ obtained through physical constraints, assumptions, or previous experiments. The likelihood $p(\mathbf{z} \mid \theb)$ is obtained through the statistical relation between the model $u$ and the data $\mathbf{z}$. Possibly the most straightforward example is $z_k = u(\theb) + \vareps_k$, where $\varepsb = \trans{(\vareps_1, \dots, \vareps_n)}$ is assumed to be multivariate Gaussian distributed with mean $\mathbf{0}$ and covariance matrix $\Sigma$. This yields the following likelihood:
\begin{equation}
	\label{eq:likelihood}
	p(\mathbf{z} \mid \theb) \propto \exp\left[ -\frac{1}{2} \trans{\mathbf{d}} \Sigma^{-1} \mathbf{d} \right], \text{with $\mathbf{d}$ a vector such that $d_k = z_k - u(\theb)$.}
\end{equation}
$\mathbf{d}$ is the so-called misfit. Bayes' law is applied to obtain the posterior $p(\theb \mid \mathbf{z})$, i.e.
\begin{equation}
	\label{eq:bayes}
	p(\theb \mid \mathbf{z}) \propto p(\mathbf{z} \mid \theb) p(\theb).
\end{equation}
The posterior PDF can be used to assess information about the parameters of the model, e.g.\ by determining the expectation or the MAP estimate (i.e.\ the maximum of the posterior). The uncertainty of these parameters can be quantified by determining the moments of the posterior PDF.

Note that the posterior depends on the likelihood, which requires an expensive evaluation of the model (see \eqref{eq:likelihood}). Therefore sampling the posterior through the application of MCMC methods \cite{Metropolis1953,Hastings1970} is typically intractable for such models.

Vector-valued models $u$ can be incorporated in this framework straightforwardly, although the likelihood requires minor modifications. Typically an observation operator is introduced that restricts $u$ to the locations where measurement data is available. We will discuss an example of this in Section~\ref{subsec:rae2822}.

Throughout this article we assume the likelihood is a continuously differentiable, Lipschitz continuous function of the misfit $\mathbf{d}$ or (more general) of the model $u$. This is true for the multivariate Gaussian likelihood and (more general) for any likelihood which has additive errors (see \cite{Kennedy2001} for more examples in the context of Bayesian model calibration). There are no further constraints on the structure of the likelihood and the prior in this work, but we do not incorporate the calibration of \emph{hyperparameters}, i.e.\ parameters introduced solely in the statistical model (an example would be the calibration of the standard deviation of $\varepsb$). Moreover, we assume the prior is not improper, i.e.\ it is a well-defined distribution with $\int_\Omega p(\theb) \dd \theb = 1$. Even though this prohibits the usage of a uniform prior on an unbounded interval, in practice our methods can be applied in such a setting.

The outline of the proposed calibration procedure is as follows. Let $u_N$ be an interpolating surrogate of $u$ using $N$ distinct nodes and model evaluations. Using $u_N$, an estimated posterior can be determined, which is used to obtain the $(N+1)^\text{th}$ node. The steps are repeated until convergence is observed. Finally MCMC can be applied to the resulting posterior, because the computationally expensive model is replaced with an explicitly known surrogate.

First, we briefly introduce the interpolation polynomial for sake of completeness. Then the nodal set we will use, the Leja nodes, will be introduced.

\subsection{Interpolation methods}
In general an interpolating polynomial can be defined as follows. Let $u: \Omega \rightarrow \mathbb{R}$ with $\Omega \subset \mathbb{R}^d$ be a continuous function. Let $D$ be given and define the set $\mathbb{P}(N, d)$ (with $N = \binom{d+D}{D}$) to be all $d$-variate polynomials of degree $D$ and lower. Using a nodal set $X_{N+1} = \{\mathbf{x}_0, \dots, \mathbf{x}_N\}$ and evaluations of $u$ at each node (i.e.\ $u(\mathbf{x}_k)$ for $k = 0, \dots, N$) the goal is to determine a polynomial $u_N \in \mathbb{P}(N, d)$ such that
\begin{equation}
	u_N(\mathbf{x}_k) = u(\mathbf{x}_k), \text{ for } k = 0, \dots, N.
\end{equation}
This construction can be extended to any $N = 1, 2, 3, \dots$, provided that the monomials of the space $\mathbb{P}(N, d)$ form a well-ordered set. Throughout this article, we use a graded reverse lexicographic order.

\subsubsection{Univariate interpolation}
In the case of $d = 1$, it is well-known that if all nodes are distinct the interpolation polynomial can be stated explicitly using Lagrange interpolating polynomials, i.e.
\begin{equation}
	\label{eq:lagrange}
	u_N(x) = (L_N u)(x) \coloneqq \sum_{k=0}^N \ell^N_k(x) u(x_k), \text{ with } \ell^N_k(x) = \prod_{\substack{j = 0 \\ j \neq k}}^N \frac{x - x_j}{x_k - x_j}.
\end{equation}
Here $L_N$ is a linear operator that yields a polynomial of degree $N$, which we will denote as $u_N$. By construction, the Lagrange basis polynomials $\ell^N_k$ have the property $\ell^N_k(x_j) = \delta_{k,j}$ (i.e.\ $\ell^N_k(x_j) = 1$ if $j = k$ and $\ell^N_k(x_j) = 0$ otherwise). Therefore $u_N(x_k) = u(x_k)$ for all $k$, such that it is indeed an interpolating polynomial.

The barycentric notation \cite{Berrut2004} can be used to numerically evaluate the interpolation polynomial given a nodal set (which is unconditionally stable \cite{Higham2004}).

\subsubsection{Multivariate interpolation}
\label{subsubsec:multivariate}
The Lagrange interpolating polynomials can be formulated in a multivariate setting, by defining them in terms of the determinant of a Vandermonde-matrix:
\begin{equation}
	\label{eq:lagrangend}
	u_N(\mathbf{x}) = (L_N u)(\mathbf{x}) \coloneqq \sum_{k=0}^N \ell^N_k(\mathbf{x}) u(\mathbf{x}_k), \text{ with } \ell^N_k(\mathbf{x}) = \frac{\det V(\mathbf{x}_0, \dots, \mathbf{x}_{k-1}, \mathbf{x}, \mathbf{x}_{k+1}, \dots, \mathbf{x}_N)}{\det V(\mathbf{x}_0, \dots, \mathbf{x}_{k-1}, \mathbf{x}_k, \mathbf{x}_{k+1}, \dots, \mathbf{x}_N)},
\end{equation}
where $V(\mathbf{x}_0, \dots, \mathbf{x}_N)$ is the $(N+1) \times (N+1)$ Vandermonde-matrix with respect to the nodal set $\{\mathbf{x}_0, \dots, \mathbf{x}_N\}$, i.e.
\begin{equation}
	\label{eq:vandermonde}
	V_{i,j}(\mathbf{x}_0, \dots, \mathbf{x}_N) = \mathbf{x}_i^{\boldsymbol\alpha_j}.
\end{equation}
Here, $\boldsymbol\alpha_j \in \mathbb{N}^d$ are defined such that for $\boldsymbol\alpha = (\alpha_1, \dots, \alpha_d)$ and $\mathbf{x} = (x_1, \dots, x_d)$, we have $\mathbf{x}^{\boldsymbol\alpha} = x_1^{\alpha_1} \cdots x_d^{\alpha_d}$. As stated before, $\boldsymbol\alpha_j$ are sorted using the graded reverse lexicographic order (i.e.\ first compare the total degree, then apply reverse lexicographic order to equal monomials). This implies that $\| \boldsymbol\alpha_j \|_1$ is a sorted sequence in $j$. Multivariate interpolation by means of this Vandermonde-matrix is only well-defined if $V$ is non-singular (then $X_{N+1}$ is called a poised interpolation sequence with respect to $\mathbb{P}(N, d)$). All nodal sequences constructed in this article are (by construction) poised.

There exist various other monomial orders, for example for the purpose to construct a sparse grid \cite{Novak1996}. Also adaptive choices have been studied \cite{Narayan2014}. Often these approaches leverage structure in the underlying distribution by decomposing it in $d$ univariate distributions (i.e.\ the distribution is ``tensorized''). Such efficient approaches cannot be applied to the context of this article, because it is rarely the case that the posterior can be decomposed in $d$ univariate distributions, due to the asymmetry in the model and the measurement data. Nonetheless, the framework and algorithms proposed in this work can easily accommodate different monomial orders.

Evaluating a multivariate interpolating polynomial numerically can be done in various ways. A commonly used approach is to rewrite \eqref{eq:vandermonde} as
\begin{equation}
	\label{eq:vandermondewithortho}
	V_{i,j}(\mathbf{x}_0, \dots, \mathbf{x}_N) = \varphi_j(\mathbf{x}_i),
\end{equation}
where $\varphi_j(\mathbf{x})$ are orthogonal basis polynomials (e.g.\ Chebyshev or Legendre polynomials). The polynomials $\varphi_j$ form a linear combination of monomials, so mathematically speaking both approaches yield the same solution.

The nodal sets used in this article use the determinant of the Vandermonde-matrix. Therefore we use a QR factorization \cite{Golub2012} of the Vandermonde-matrix to determine the interpolating polynomial and reuse the QR factorization to determine the nodal set (see Section~\ref{subsec:weightedlejanodes} for details).

\subsection{Weighted Leja nodes}
\label{subsec:weightedlejanodes}
For the purpose of Bayesian model calibration, we desire an algorithm to determine a nodal set $X_N$ for any $N$ having the following properties:
\begin{enumerate}
	\item \textbf{Accuracy:} the nodal sets should yield an accurate posterior. We are mainly interested in estimating the posterior, i.e.\ it is not strictly necessary to have an accurate surrogate model on the full domain.
	\item \textbf{Nested:} we require $X_i \subset X_j$ for $i < j$, such that the obtained interpolant can be refined by reusing existing model evaluations.
	\item \textbf{Weighting:} the goal is to determine the next node based on the posterior obtained so far. The algorithm of the nodal set should allow for this.
\end{enumerate}
In this work we consider weighted Leja nodes, which form a sequence of nodes. The sequence is therefore by definition nested. First, we will define the univariate Leja nodes and generalize those to multivariate Leja nodes.

The definition of weighted Leja nodes is by induction. Let $\rho: \mathbb{R} \rightarrow \mathbb{R}_+$ be a bounded PDF (with $\mathbb{R}_+ \coloneqq [0, \infty)$) and let $\{x_0, \dots, x_N\}$ be a sequence of Leja nodes. Then the next node is defined by maximizing the numerator of $\ell^{N+1}_{N+1}(x)$, i.e.
\begin{equation}
	\label{eq:leja1}
	x_{N+1} \coloneqq \argmax_{x \in \mathbb{R}} \rho(x) |x - x_0| |x - x_1| \cdots |x - x_N|.
\end{equation}
This maximization problem does not necessarily have a unique solution. To ensure that a solution exists, it is necessary to assume that either $\rho(x)$ has bounded and closed support or (more generally) that the polynomials are dense in the space equipped with the $\infty$-norm weighted with $\rho(x)$, i.e.\ the norm $\| f \|_\rho = \| f \, \rho \|_\infty = \sup_{x \in \mathbb{R}} |f(x) \rho(x) |$ \cite{Narayan2014}. Note that the former implies the latter and that $\rho$ has finite moments in these cases. If there are multiple values maximizing \eqref{eq:leja1}, we pick the one with smallest $x$ to ensure that the sequence is reproducible (in multivariate spaces, we select the smallest one using a lexicographic ordering). The initial node $x_0$ is defined as the smallest global maximum of $\rho(x)$.

If $\rho(x)$ has finite moments, it decays faster than any polynomial grows for $x \to \infty$ (which makes the maximization problem above well-defined). To see this, assume $\rho(x)$ decays slower than the polynomial $x^k$ grows for $k > 1$, or equivalently assume $\rho(x) > 1/x^k$ for $x > A$. Then
\begin{equation}
	\int_A^\infty x^k \rho(x) \dd x > \int_A^\infty x^k \frac{1}{x^k} \dd x = \infty,
\end{equation}
which cannot be the case as $\rho(x)$ has finite moments.

Notice that definition \eqref{eq:leja1} can be rewritten as follows:
\begin{equation}
	\label{eq:leja1log}
	x_{N+1} = \argmax_{x \in \mathbb{R}, \rho(x) > 0} \log \rho(x) + \sum_{k=0}^N \log |x - x_k|.
\end{equation}
If $\rho(x)$ is bounded from below and above, i.e.\ $A \leq \rho(x) \leq B$ for $0 < A < B$ and $\rho(x)$ has bounded support, the sum $\log |x - x_0| + \dots + \log |x - x_N|$ will dominate the maximal value for large $N$. Hence for any $x$ the value of the sum will increase (but remains bounded as $\rho(x)$ has bounded support) and $\rho(x)$ will remain constant (as $\rho$ is independent of $N$). This implies that for $\rho(x)$ that are bounded from below and above, the influence of the weighting function decreases as $N$ increases.

Unweighted Leja nodes are defined with the uniform weighting function on $[-1, 1]$. We want to emphasize that multiplying the weighting function with a constant yields an identical sequence. This property is very useful for our purposes, as it allows us to neglect the constant of proportionality (often called the evidence) in Bayes' law (see \eqref{eq:bayes}).

\begin{figure}
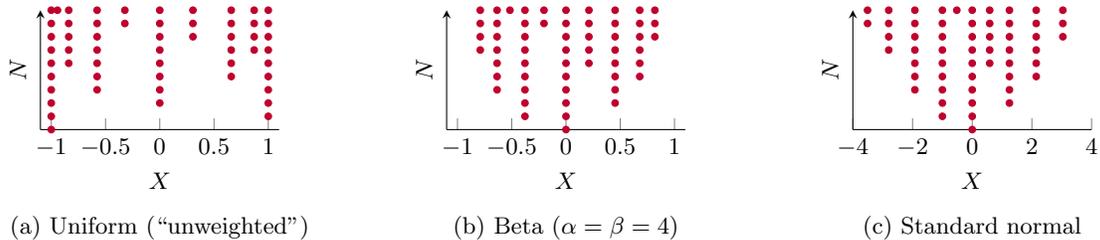

	\centering
	\begin{minipage}{.32\textwidth}
		\small
		\centering
		\includepgf{.9\textwidth}{0.6\textwidth}{leja.uniform.tikz}
		\subcaption{Uniform (``unweighted'')}
		\label{subfig:leja-uniform}
	\end{minipage}
	\begin{minipage}{.32\textwidth}
		\small
		\centering
		\includepgf{.9\textwidth}{0.6\textwidth}{leja.beta.tikz}
		\subcaption{Beta ($\alpha = \beta = 4$)}
	\end{minipage}
	\begin{minipage}{.32\textwidth}
		\small
		\centering
		\includepgf{.9\textwidth}{0.6\textwidth}{leja.stdgauss.tikz}
		\subcaption{Standard normal}
	\end{minipage}

	\caption{Univariate Leja sequences for various number of nodes and various well-known distributions.}
	\label{fig:leja}
\end{figure}

Examples of these sequences are depicted in Figure~\ref{fig:leja}. Throughout this article, univariate Leja nodes are determined by applying Newton's method to the derivative of the logarithm of the maximization problem above, i.e.\ \eqref{eq:leja1log} is solved instead of \eqref{eq:leja1}. By determining all local maxima between two consecutive nodes in parallel, large numbers of nodes can be calculated fast and accurately (as the maximization function is smooth between two nodes). Numerical cancellation is kept minimal by using extended precision arithmetic (with machine epsilon approximately $10^{-19}$).

The definition of univariate Leja nodes can be generalized to a multidimensional setting in a similar way as we did in Section~\ref{subsubsec:multivariate} with interpolation. To this end, let $\rho: \mathbb{R}^d \rightarrow \mathbb{R}_+$ be a multivariate PDF. Let $\mathbf{x}_0 \in \mathbb{R}^d$ be an initial node with $\rho(\mathbf{x}_0) > 0$. Then given the nodes $\mathbf{x}_0, \dots, \mathbf{x}_N$, the next node $\mathbf{x}_{N+1}$ is defined as follows:
\begin{equation}
	\label{eq:leja2}
	\mathbf{x}_{N+1} \coloneqq \argmax_{\mathbf{x} \in \mathbb{R}^d} \rho(\mathbf{x}) \left| \det V(\mathbf{x}_0, \dots, \mathbf{x}_N, \mathbf{x}) \right|.
\end{equation}
Here, $V$ is the Vandermonde-matrix defined in Section~\ref{subsubsec:multivariate}. The absolute value of the determinant of $V$ is independent of the set of polynomials that is used to construct $V$, so the definition is mathematically the same for both monomials and orthogonal polynomials.

Determining multivariate Leja nodes is less trivial compared to univariate nodes and is typically done by randomly (or quasi-randomly) sampling the space of interest and selecting the node that results in the highest determinant. It is significantly more complicated to reliably apply Newton's method in this case, as the space cannot be easily partioned in regions where the local maxima reside. To reach a comparable accuracy, it is important to be able to use a large number of samples, so it must be possible to calculate the determinant fast. We suggest to calculate the determinant by an extended QR factorization of the $(N+2) \times (N+1)$-matrix $V(\mathbf{x}_0, \dots, \mathbf{x}_N)$ and to add the column containing $\mathbf{x}$ by applying a rank-1 update. If a QR factorization has been calculated to determine the interpolating polynomial (see Section~\ref{subsubsec:multivariate}), it can be reused here. As a rank-1 update is an efficient procedure, a large number of samples can be used and therefore we assume that the approximation error is negligible in this case. Examples of Leja sequences defined by \eqref{eq:leja2} can be found in Figure~\ref{fig:multileja}.

\begin{figure}
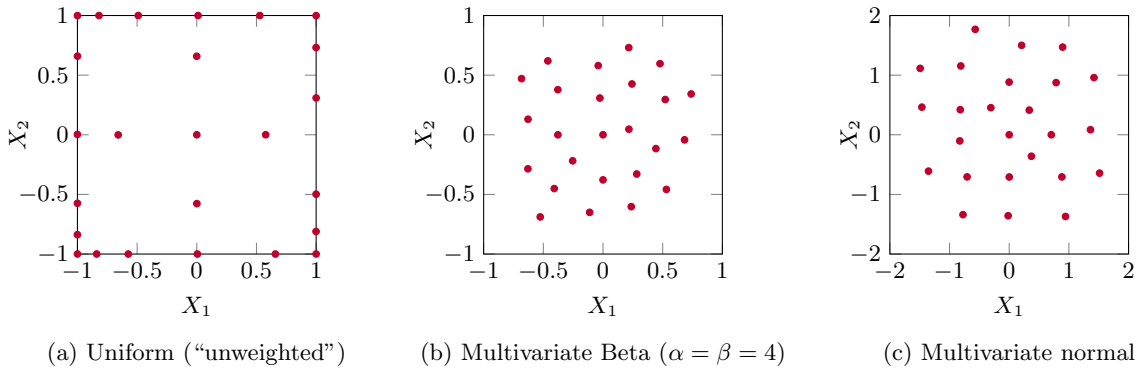

	\centering
	\begin{minipage}{.32\textwidth}
		\small
		\centering
		\includepgf{.9\textwidth}{.9\textwidth}{leja2.uniform.tikz}
		\subcaption{Uniform (``unweighted'')}
	\end{minipage}
	\begin{minipage}{.32\textwidth}
		\small
		\centering
		\includepgf{.9\textwidth}{.9\textwidth}{leja2.beta.tikz}
		\subcaption{Multivariate Beta ($\alpha = \beta = 4$)}
	\end{minipage}
	\begin{minipage}{.32\textwidth}
		\small
		\centering
		\includepgf{.9\textwidth}{.9\textwidth}{leja2.stdgauss.tikz}
		\subcaption{Multivariate normal}
	\end{minipage}

	\caption{Multivariate Leja sequences of 25 nodes using various well-known distributions.}
	\label{fig:multileja}
\end{figure}

\subsection{Calibration using Leja nodes}
\label{subsec:weightingfnc}
In this section we will derive a weighting function to be used in the interpolation procedure discussed in the previous section, with the goal to approximate the posterior. Theoretical details are provided in Section~\ref{sec:convergence}. First we discuss the rationale behind the weighting function in Section~\ref{subsubsec:rationale}. The weighting function itself is presented in Section~\ref{subsubsec:weightingfunc} and the mathematical derivation it is based upon is presented in Section~\ref{subsubsec:mvt}. The weighting function has one free parameter, which is discussed in more detail in Section~\ref{subsubsec:zeta}.

\subsubsection{Rationale}
\label{subsubsec:rationale}
If the posterior is known explicitly and samples can be readily drawn from it, it is possible to determine weighted Leja nodes with weighting function $\rho(\theb) = p(\theb \mid \mathbf{z})$. These nodes provide an interpolant that is very suitable for evaluating integrals with respect to the posterior (this is commonly known as Bayesian prediction). However, the posterior is generally not explicitly available because it depends on the model $u$, which in itself is not known explicitly and can only be determined on (finitely many) nodes. Therefore the need arises for an interpolation sequence that approximates $u$ such that the posterior determined with this approximation is accurate.

To this end, let $p_N(\theb \mid \mathbf{z})$ be the posterior determined using $u_N(\theb)$, i.e.\ the interpolant of $u$ using $N+1$ nodes in $\theb$. If the likelihood is according to \eqref{eq:likelihood}, $p_N$ is as follows:
\begin{equation}
	p_N(\theb \mid \mathbf{z}) \propto p(\theb) \exp\left[ -\frac{1}{2} \trans{\mathbf{d}} \Sigma^{-1} \mathbf{d} \right], \text{with $\mathbf{d}$ a vector such that $d_k = z_k - u_N(\theb)$}.
\end{equation}
We will use the definition of the weighted Leja nodes from \eqref{eq:leja2} to determine the next node. The natural idea is to construct $p_{N+1}(\theb \mid \mathbf{z})$ (i.e.\ a new approximation of the posterior) by determining a new weighted Leja node using $p_N(\theb \mid \mathbf{z})$ (i.e.\ the existing approximation of the posterior). Such a sequence can be numerically unstable, because it solely places nodes in regions where the approximate posterior is high and therefore yields spurious oscillations in other regions in the domain.

The key idea is to balance between the accuracy of the interpolant and the accuracy of the posterior. There are various methods to do this, but we choose to temper the effect of the (possibly inaccurate) approximate posterior by adding a constant value $\zeta$ to it. The higher this $\zeta$, the more the posterior tends to the prior. In Section~\ref{subsec:accuracyofnodalsets} it is demonstrated that for any $\zeta > 0$, the interpolant constructed with these weighted Leja nodes has (at least) the same asymptotic convergence rate as an interpolant determined with weighted Leja nodes without adaptivity. If $\zeta$ is chosen correctly, the approximate posterior is already accurate for moderately small $N$.

\subsubsection{The adaptive weighting function}
\label{subsubsec:weightingfunc}
To introduce this construction formally, we assume that a function $P: \mathbb{R} \rightarrow \mathbb{R}_+$ exists such that
\begin{equation}
	\label{eq:definitionofP}
	p(\mathbf{z} \mid \theb) = P(u(\theb)),
\end{equation}
where $P$ typically is a PDF which follows from the statistical model. In the example discussed in \eqref{eq:likelihood} $P$ is a Gaussian PDF, i.e.
\begin{equation}
	\label{eq:gaussianP}
	P: \mathbb{R} \to \mathbb{R}_+, \text{ with } P(u) \propto \exp\left[-\frac{1}{2} \trans{\mathbf{d}} \Sigma^{-1} \mathbf{d}\right]\text{ and } d_k = z_k - u.
\end{equation}

We assume that the function $P$ is globally Lipschitz continuous and continuously differentiable. Many statistical models used in a statistical setting yield Lipschitz continuous $P$, because a bounded continuously differentiable function $P(u)$ with $P'(u) \to 0$ for $u \to \pm \infty$ is Lipschitz continuous. The domain of definition of $P$ is the image of the model $u(\theb)$, so functions $P$ that are only Lipschitz continuous in the set described by the image of $u(\theb)$ also fit in this framework (for example the Gamma distribution on the positive real axis).

The weighting function proposed in this article, called $q_N$, clearly depends on $N$:
\begin{equation}
	\label{eq:definitionofQ}
	q_N(\theb \mid \mathbf{z}) = |P'(u_N(\theb))| p(\theb) + \zeta p(\theb), \text{ where $\zeta > 0$ is a free parameter}.
\end{equation}
So, if $\theb_0, \dots, \theb_N$ are the first $N+1$ Leja nodes, $\theb_{N+1}$ is determined as follows:
\begin{equation}
	\label{eq:lejaadaptive}
	\theb_{N+1} = \argmax_\theb q_N(\theb \mid \mathbf{z}) \left| \det V(\theb_0, \dots, \theb_N, \theb) \right|.
\end{equation}
Here the derivative $P'$ is with respect to $u$, i.e.
\begin{equation}
	P'(u(\theb)) = \frac{\partial P}{\partial u} (u(\theb)).
\end{equation}
We want to emphasize that for the evaluation of $P'(u_N(\theb))$ no costly evaluation of the full model $u$ is necessary, since $P'$ is independent of $u$. In the example from \eqref{eq:likelihood}, $P'$ becomes the following:
\begin{equation}
	\label{eq:dPdu}
	P': \mathbb{R} \to \mathbb{R}, \text{ with } P'(u) \propto -\frac{1}{2} \left(\trans{\mathbf{1}} \Sigma^{-1} \mathbf{d} + \trans{\mathbf{d}} \Sigma^{-1} \mathbf{1}\right) \exp\left[ -\frac{1}{2} \trans{\mathbf{d}} \Sigma^{-1} \mathbf{d}\right] \text{ and } d_k = z_k - u,
\end{equation}
with $\mathbf{1} = \trans{(1, 1, \dots, 1)} \in \mathbb{R}^n$.

\subsubsection{Mean value theorem}
\label{subsubsec:mvt}
The weighting function $q_N$ as defined in \eqref{eq:definitionofQ} follows naturally by applying the mean value theorem to the error of the approximate posterior. This introduces the derivative $P'$ in the expression. To this end, let a fixed $\theb$ be given, and apply the mean value theorem as follows:
\begin{equation}
\label{eq:mvt}
\begin{aligned}
	|p_N(\theb \mid \mathbf{z}) - p(\theb \mid \mathbf{z})| &= |P(u_N(\theb)) - P(u(\theb))| p(\theb) \\
	&= |P'(\xi)| |u_N(\theb) - u(\theb)| p(\theb) \\
	&= |P'(u_N(\theb)) + \zeta_\theb| |u_N(\theb) - u(\theb)| p(\theb),
\end{aligned}
\end{equation}
with $\xi$ an (unknown) value between $u_N(\theb)$ and $u(\theb)$ and $\zeta_\theb = P'(\xi) - P'(u_N(\theb))$. Essentially $\zeta_\theb$ is used to represent higher order derivatives of $P$ in this expression. The value of $\zeta_\theb$ depends on $\theb$ and on the model $u$, which is not explicitly known. By further expanding $P'$, it can be shown that $\zeta_\theb$ scales with $|u_N(\theb) - u(\theb)|$, provided that $P$ is twice differentiable with bounded second order derivative:
\begin{equation}
	\zeta_\theb = P'(\xi) - P'(u_N(\theb)) = \frac{1}{2} P''(\widehat{\xi})(u_N(\theb) - u(\theb)), \text{ for a $\widehat{\xi}$ between $P'(u_N(\theb))$ and $P'(u(\theb))$}.
\end{equation}
Hence if $u_N(\theb) \to u(\theb)$ for $N \to \infty$ and $P''$ bounded (or: the divided difference of $P'$ is bounded), it holds that $\zeta_\theb \to 0$ for $N \to \infty$. In this work, the constant $\zeta_\theb$ is used to measure how far the likelihood of the interpolant is from the likelihood of the true model. The idea is to add a Leja node $\theb_{N+1}$ where the error in the posterior is large, though such that the interpolant remains stable. The weighting function $q_N$ as introduced before follows by taking the $\infty$-norm in $\theb$ on both sides of \eqref{eq:mvt}:
\begin{equation}
\label{eq:mvtinf}
\begin{aligned}
	\| p_N(\theb \mid \mathbf{z}) - p(\theb \mid \mathbf{z}) \|_\infty &= \| P(u_N(\theb)) - P(u(\theb)) \|_\infty \\
	&= \| |P'(\xi)| (u_N(\theb) - u(\theb)) p(\theb) \|_\infty \\
	&\leq \| \bigl( |P'(u_N(\theb))| + \zeta \bigr) \bigl( u_N(\theb) - u(\theb) \bigr) p(\theb) \|_\infty \\
	&= \| (u_N(\theb) - u(\theb)) q_N(\theb) \|_\infty,
\end{aligned}
\end{equation}
with $\zeta \geq |\zeta_\theb| = |P'(\xi) - P'(u_N(\theb))|$ for all $\theb$.

The algorithm proposed in this article is to (iteratively) firstly determine $q_N$, secondly determine $\theb_{N+1}$ using \eqref{eq:lejaadaptive}, and finally determine $u(\theb_{N+1})$ and reconstruct the interpolant (which yields $u_{N+1}$ and consequently $p_{N+1}(\theb \mid \mathbf{z})$). This algorithm is sketched in Figure~\ref{fig:algorithm}. Convergence can be assessed in various ways, for example using the $\infty$-norm or the Kullback--Leibler divergence. We will use the $\infty$-norm, as determining the Kullback--Leibler divergence in higher-dimensional spaces is numerically challenging.

\begin{figure}[t]
	\centering
	\footnotesize
	
	\def\svgwidth{.8\textwidth}
	\input{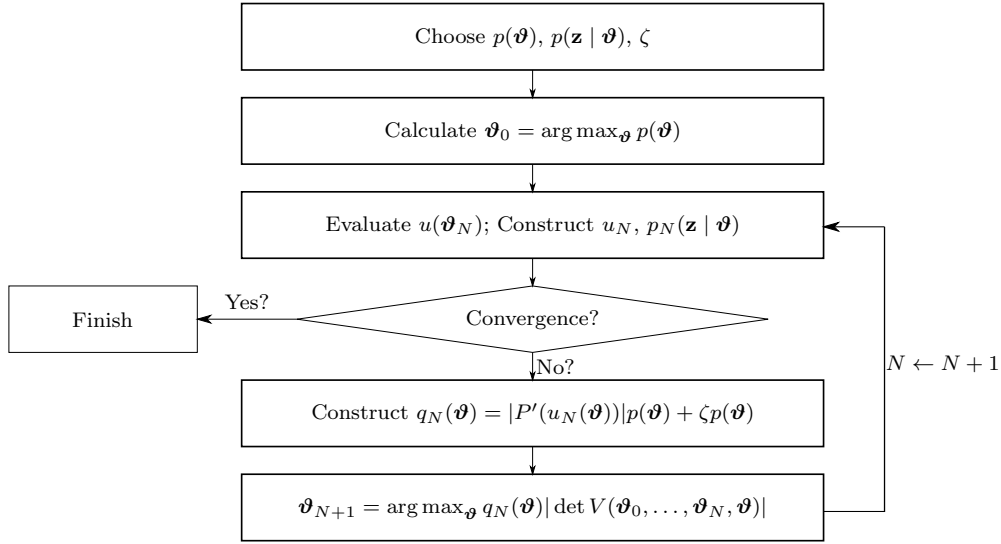}

	\caption{Schematic overview of the algorithm proposed in this article.}
	\label{fig:algorithm}
\end{figure}

The exact value of $\zeta_\theb$ is not known a priori and depends on $\theb$. Nonetheless, we will demonstrate that for any $\zeta > 0$ it holds that $\| u - u_N \|_\infty \to 0$ (for $N \to \infty$), provided that ``conventional'' weighted Leja nodes produce a converging interpolant. If $u_N \to u$ for $N \to \infty$, the exact value of $\zeta$ converges to $0$, hence any value of $\zeta$ will work for sufficiently large $N$. We will further study the convergence of this method in Section~\ref{sec:convergence}.

\begin{figure}[t]
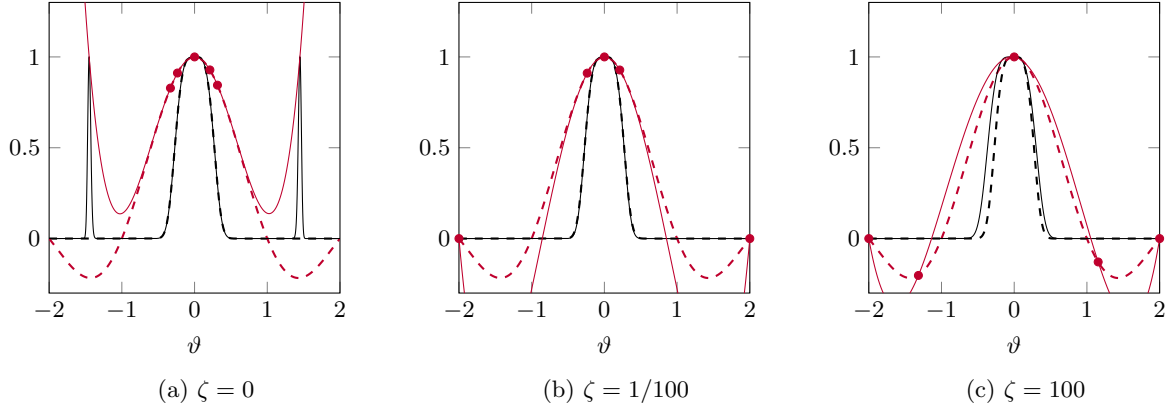

	\begin{minipage}{.33\textwidth}
		\includepgf{\textwidth}{\textwidth}{example-direct-leja-0.tikz}
		\subcaption{$\zeta = 0$}
	\end{minipage}%
	\begin{minipage}{.33\textwidth}
		\includepgf{\textwidth}{\textwidth}{example-direct-leja-1-100.tikz}
		\subcaption{$\zeta = 1/100$}
	\end{minipage}%
	\begin{minipage}{.33\textwidth}
		\includepgf{\textwidth}{\textwidth}{example-direct-leja-100.tikz}
		\subcaption{$\zeta = 100$}
		\label{subfig:example-direct-leja-100}
	\end{minipage}%

	\caption{Interpolation of the $\mathrm{sinc}$ function using 5 weighted Leja nodes with respect to the posterior using a tempering parameter $\zeta$. The model $u(\vartheta)$ and (unscaled) posterior $p(\vartheta \mid z_1)$ are depicted in red and black respectively. The solid line represents the result constructed by means of interpolation and the ``true'' model and posterior are depicted using a dashed line.}
	\label{fig:example-direct-leja}
\end{figure}

\subsubsection{\texorpdfstring{Choice of $\boldsymbol\zeta$}{Choice of zeta}}
\label{subsubsec:zeta}
To illustrate the behavior of weighting function \eqref{eq:definitionofQ}, examples of interpolants obtained using Leja nodes weighted using $q_N$ in conjunction with the exact posterior are depicted in Figure~\ref{fig:example-direct-leja}. Here the parameter $\vartheta$ of the univariate function $u(\vartheta) = \mathrm{sinc}(\vartheta) = \sin(\vartheta) / \vartheta$ is ``calibrated'' using the Gaussian likelihood from \eqref{eq:gaussianP} with $\sigma = 1/10$, a uniform prior defined on $[-2, 2]$, and one data point at $z_1 = 1$. Hence the exact posterior is as follows:
\begin{equation}
	p(\vartheta \mid z_1) \propto \begin{cases}
		\exp\left[ -\frac{1}{2 \sigma^2} |u(\vartheta) - z_1|^2 \right] &\text{if $|\vartheta| \leq 2$}, \\
		0 &\text{otherwise}.
	\end{cases}
\end{equation}
The weighting function under consideration is $q_N(\vartheta) = |P'(u(\vartheta))| + \zeta$, where $u$ is used instead of $u_N$ to illustrate the effect of $\zeta$.

If $\zeta = 0$ (no tempering) the interpolant is indeed accurate with respect to the posterior (i.e.\ the weighted $p(\vartheta \mid z_1)$-norm), but yields an incorrect approximate posterior because the interpolant crosses the value of the data incorrectly around $\vartheta = \pm 1.5$. These spurious oscillations disappear for larger $N$, but for different test cases this is not necessarily the case (as it requires global analyticity). For $\zeta = 100$, it is guaranteed that $\zeta \geq |P'(\xi) - P'(u_N(\vartheta))|$ for all $\vartheta$, but the nodes determined with that value are, due to the large variations in the determinant of the Vandermonde-matrix, not sensitive to small variations in the approximate posterior, and are therefore pointwise close to unweighted Leja nodes (e.g.\ compare Figure~\ref{subfig:example-direct-leja-100} with Figure~\ref{subfig:leja-uniform}). The best strategy is to take a small non-zero value of $\zeta$, which balances posterior accuracy with stability. For such a small non-zero value, the second and third node are basically unweighted Leja nodes (and end up on the boundary). This does demonstrate the importance of tempering on the effect of the approximate posterior, which becomes especially important if the function $u$ is not globally analytic (but ``only'' continuous).

The key point in obtaining a converging interpolant is that $\zeta > 0$. If $\zeta = 0$, the inaccuracy of $u_N$ can significantly deteriorate the convergence (see Figure~\ref{fig:example-direct-leja}), except possibly if $u$ is globally analytic. If the goal is to optimize $\zeta$, we suggest a heuristically adaptive approach. Start with $\zeta = \zeta_0 > 0$ and for each iteration, multiply $\zeta$ with a constant $k > 1$ if the error in the posterior increases and divide $\zeta$ by $k$ if the interpolation error decreases. The error can be estimated by comparing two consecutive approximate posteriors. This procedure is however not necessary to obtain convergence for the examples in this article, for which a fixed value of $\zeta$ is sufficient.

\section{Convergence of the posterior}
\label{sec:convergence}
In this section the convergence of the estimated posterior to the true posterior is studied, denoted as follows:
\begin{equation}
	\| p_N(\theb \mid \mathbf{z}) - p(\theb \mid \mathbf{z}) \|_\infty \to 0, \text{ for $N \to \infty$}.
\end{equation}
It is difficult to theoretically demonstrate that this is case, since the convergence rate of interpolants constructed with Leja nodes is only known in some specific cases. However, we will demonstrate that the convergence rate of an interpolant determined with adaptively weighted Leja nodes is similar to one determined with Leja nodes without adaptivity, such that all results on the convergence of these conventional Leja nodes carry over.

The analysis is split into two parts. First, in Sections~\ref{subsec:lebesgue} and \ref{subsec:accuracyofnodalsets} the focus is on the model, i.e.\ it is assessed in which cases $\| u_N - u \|_{p(\theb)} \to 0$ for $N \to \infty$ (where $p(\theb)$ denotes the prior). In Section~\ref{subsec:lebesgue} convergence properties of interpolation methods are briefly reviewed. In Section~\ref{subsec:accuracyofnodalsets} the focus is specifically on Leja nodes, a case that will be assessed numerically. Moreover, the close relation between adaptively weighted Leja nodes and Leja nodes without adaptivity is considered.

The second part of the analysis consists of demonstrating that the posterior converges if the interpolant converges. Specifically, in Section~\ref{subsec:convergence} the following is demonstrated:
\begin{equation}
	\| p_N(\theb \mid \mathbf{z}) - p(\theb \mid \mathbf{z}) \|_\infty \leq D \| u_N - u \|_{p(\theb)}, \text{ with $D$ a constant independent of $N$}.
\end{equation}
The conventional way of describing the distance between two distributions is by means of the Kullback--Leibler divergence. In Section~\ref{subsec:KLconvergence} it is proved that if the interpolant converges to the true model, the Kullback--Leibler divergence between the approximate posterior and the true posterior converges to zero. Moreover, the rate of convergence doubles.

\subsection{Accuracy of interpolation methods}
\label{subsec:lebesgue}
The accuracy of interpolation methods can be assessed in two ways: using \emph{pointwise} error bounds which are typically based on Taylor expansions and \emph{global} error bounds which are typically based on the Lebesgue inequality. We will use the latter type.

Let $\mathbb{P}(K) = \mathbb{P}(K, 1)$, i.e.\ all univariate polynomials of degree less than or equal to $K$. We assume $u \in C[-1, 1]$ (i.e.\ a continuous function defined in $[-1, 1]$) and $\| u \|_\infty < \infty$ if not stated otherwise. It is well-known that $C[-1, 1]$ equipped with the norm $\| \cdot \|_\infty$ forms a Banach space.

Let $X_N = \{x_0, \dots, x_N\} \subset [-1, 1]$ be a set of interpolation nodes and let $L_N: C[-1,1] \rightarrow \mathbb{P}(N)$ be the Lagrange interpolation operator (see \eqref{eq:lagrange}) that determines the interpolating polynomial given the nodal set $X_N$. Then for any polynomial $\varphi_N$ of degree $N$ we have
\begin{equation}
	\| L_N u - u \|_\infty \leq (1 + \| L_N \|_\infty) \| \varphi_N - u \|_\infty,
\end{equation}
where $\Lambda_N \coloneqq \| L_N \|_\infty = \sup_{x \in [-1, 1]} \sum_{k=0}^N |\ell^N_k(x)|$ is the operator norm of $L_N$ induced by the norm $\| \cdot \|_\infty$ discussed above. $\Lambda_N$ is called the Lebesgue constant \cite{Ibrahimoglu2016}. The inequality holds for any polynomial $\varphi_N$, so an immediate result is the Lebesgue inequality:
\begin{equation}
	\label{eq:lebesgue}
	\| L_N u - u \|_\infty \leq (1 + \Lambda_N) \inf_{\varphi_N \in \mathbb{P}(N)} \| \varphi_N - u \|_\infty.
\end{equation}

The obtained expression contains a part that solely depends on the nodes, i.e.\ $(1 + \Lambda_N)$, and a part that solely depends on the function $u$, i.e.\ $\inf_{\varphi_N \in \mathbb{P}(N)} \| \varphi_N - u \|_\infty$. The second part is commonly known as the best approximation error. In the procedure used for calibration, nodes are determined using a weighting function $\rho$. To assess the accuracy of nodes that use weighting, we reconsider the weighted $\infty$-norm $\| u \|_\rho = \| \rho \, u \|_\infty$. Here, we assume that $\rho(x): \Omega \rightarrow \mathbb{R}_+$ is a bounded PDF\footnote{The boundedness is not strictly necessary for $C(\Omega)$ to be a Banach space, but for the cases discussed in this article boundedness suits.}, such that that $C(\Omega)$ equipped with the norm $\| \cdot \|_\rho$ forms a Banach space. The support of $u$ (say $\Omega$) is allowed to be unbounded, in contrast to the unweighted case. The unweighted case is a special case of the weighted case.

Using this norm we can derive a similar estimate as \eqref{eq:lebesgue} by introducing \cite{Jantsch2016}:
\begin{equation}
	\Lambda^\rho_N \coloneqq \| L_N \|_\rho = \sup_{x \in \mathbb{R}} \sum_{k=0}^N \frac{\rho(x)}{\rho(x_k)} |\ell^N_k(x)|.
\end{equation}
Here, $\Lambda_N^\rho$ is called the weighted Lebesgue constant, i.e.\ the norm of the operator $u \rightarrow \rho L_N (u / \rho)$. We call the result the weighted Lebesgue inequality:
\begin{equation}
	\label{eq:wlebesgue}
	\| L_N u - u \|_\rho \leq (1 + \Lambda_N^\rho) \inf_{\varphi_N \in \mathbb{P}(N)} \| \varphi_N - u \|_\rho.
\end{equation}

The Lebesgue inequality does not readily provide means to estimate the order of convergence. For this purpose, Jackson's inequality~\cite{Jackson1982,Passow1970} can be used, that relates the best approximation error to the modulus of continuity of $u$. Important for this work is that if $u$ is Lipschitz continuous (or continuous and bounded on a compact domain) then a sublinearly growing Lebesgue constant provides a converging interpolant. Various other results on this topic exist, the interested reader is referred to the accessible introduction in the book of \citet{Watson1980} and the references therein for more information.

\subsection{Lebesgue constant of Leja nodes}
\label{subsec:accuracyofnodalsets}
\begin{figure}
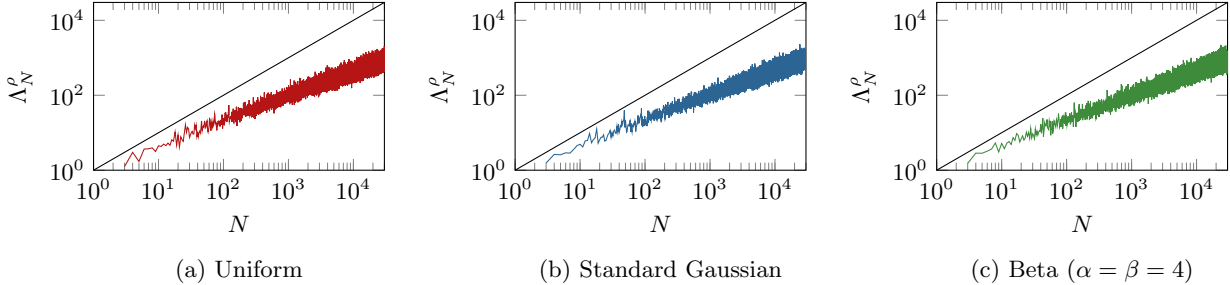

	\begin{minipage}{.33\textwidth}
		\centering
		\small
		\includepgf{\textwidth}{.7\textwidth}{lebesgue-1.tikz}
		\subcaption{Uniform}
		\label{fig:lebesgue-uniform}
	\end{minipage}~
	\begin{minipage}{.33\textwidth}
		\centering
		\small
		\includepgf{\textwidth}{.7\textwidth}{lebesgue-2.tikz}
		\subcaption{Standard Gaussian}
		\label{fig:lebesgue-stdgauss}
	\end{minipage}~
	\begin{minipage}{.33\textwidth}
		\centering
		\small
		\includepgf{\textwidth}{.7\textwidth}{lebesgue-3.tikz}
		\subcaption{Beta ($\alpha=\beta=4$)}
		\label{fig:lebesgue-beta}
	\end{minipage}

	\caption{The weighted Lebesgue constant of weighted Leja nodes for three different distributions. The solid line depicts $\Lambda_N = N$.}
	\label{fig:lebesgue}
\end{figure}
It is both an advantage and a disadvantage that the Lebesgue constant solely depends on the nodal set: we do not have to take the model into account to estimate the accuracy, but the resulting estimate does not leverage any properties of the model. The algorithm discussed in Section~\ref{subsec:weightingfnc} does use the model and therefore cannot be fit directly in the framework set out in the previous section.

Many nodal sets exist with a logarithmically growing Lebesgue constant, which is asymptotically the optimal growth. For example, Chebyshev nodes (i.e.\ the nodes from the Clenshaw--Curtis quadrature rule) have $\Lambda_N = \mathcal{O}(\log N)$ \cite{Ibrahimoglu2016}. Moreover, we already stated that the Chebyshev nodes are nested such that the nodes for $N = 2^l+1$ (for integer $l$) are contained in the nodes for $N = 2^{l+1}+1$. However, the Chebyshev nodes are only defined in an \emph{unweighted} setting. Another well-known example are equidistant nodes, which have an exponentially growing Lebesgue constant. This can be observed by interpolating Runge's function.

Although it is known that the Lebesgue constant of both weighted and unweighted Leja sequences grows sub-exponentially \cite{Taylor2008,Jantsch2016,Taylor2010}, the exact growth (or a strict upper bound) is not known. We have therefore numerically determined the Lebesgue constant for $N$ up to $30\,000$ for several distributions and observed $\Lambda_N^\rho \leq N$, see Figure~\ref{fig:lebesgue}. For the purpose of these figures, the Leja nodes and the Lebesgue constant have been determined by applying Newton's method to the derivative of the function, as described in Section~\ref{subsec:weightedlejanodes}.

\begin{figure}
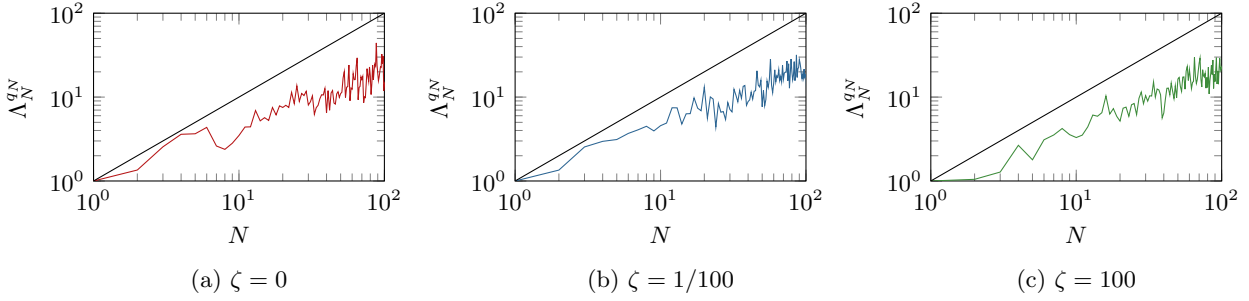

	\begin{minipage}{.33\textwidth}
		\centering
		\small
		\includepgf{\textwidth}{.7\textwidth}{lebesgue-adapt-0.tikz}
		\subcaption{$\zeta = 0$}
		\label{fig:lebesgue-adapt-0}
	\end{minipage}~
	\begin{minipage}{.33\textwidth}
		\centering
		\small
		\includepgf{\textwidth}{.7\textwidth}{lebesgue-adapt-0-01.tikz}
		\subcaption{$\zeta = 1/100$}
		\label{fig:lebesgue-adapt-0-01}
	\end{minipage}
	\begin{minipage}{.33\textwidth}
		\centering
		\small
		\includepgf{\textwidth}{.7\textwidth}{lebesgue-adapt-100.tikz}
		\subcaption{$\zeta = 100$}
		\label{fig:lebesgue-adapt-100}
	\end{minipage}

	\caption{The weighted Lebesgue constant of adaptively weighted Leja nodes using a Gaussian likelihood with $\sigma=1/10$, a uniform prior on $[-2, 2]$, a single data point $z_1=1$, and the function $u(\vartheta) = \mathrm{sinc}(\vartheta)$. The solid line depicts $\Lambda_N = N$.}
	\label{fig:lebesgue-adapt}
\end{figure}
Except for some specific cases (such as the unit disk \cite{Calvi2011}), not much is known about the Lebesgue constant in the multivariate case. Unfortunately, as far as the authors know, it is difficult to determine the Lebesgue constant of multivariate Leja nodes accurately enough to create a similar plot as Figure~\ref{fig:lebesgue}. For small number of nodes ($N \lesssim 100$) and low dimensionality ($d \lesssim 5$) the Lebesgue constants seem to grow similarly, but the sampling procedure significantly deteriorates the accuracy. Moreover this number is too small to draw conclusions about general asymptotic behavior. Nonetheless, the numerically determined growth of the Lebesgue constant is sufficient for the purposes in the current article (as $N \lesssim 100$ throughout this article). We want to emphasize that contrary to the Lebesgue constant, large numbers of multivariate Leja nodes can be determined efficiently by means of sampling, as evaluating the maximization function is significantly more straightforward (see Section~\ref{subsec:weightedlejanodes}).

These results do not carry over straightforwardly to the case of adaptively weighted Leja nodes where the weighting function depends on the number of nodes. However, for reasonably small $N$ the Lebesgue constant can be assessed numerically. To this end, we have determined the Lebesgue constant $\Lambda^{q_N}_N$, i.e.\ the Lebesgue constant weighted with $q_N$ from \eqref{eq:definitionofQ}, of adaptively weighted Leja nodes using the aforementioned example of a Gaussian likelihood in conjunction with the function $u(\vartheta) = \mathrm{sinc}(\vartheta)$ (see Figure~\ref{fig:lebesgue-adapt}). Determining Leja nodes in this case is still relatively straightforward, but determining the Lebesgue constant accurately is significantly less trivial due to the varying weight function, so we limit ourselves to 100 nodes. Even though the weighting function now depends on the number of nodes, it appears that the Lebesgue constant still grows sublinearly. This result, in conjunction with the numerical results from Section~\ref{sec:numerics}, indicates that weighted Leja nodes as proposed in this article indeed yield an interpolant that can be used to construct an accurate approximate posterior. Notice that slow growth of $\Lambda^{q_N}_N$ implies slow growth of $\Lambda^p_N$ and vice versa (where $p$ denotes the prior), which can be seen as follows:
\begin{equation}
	\label{eq:adaptivebound}
	\begin{aligned}
		\Lambda^p_N &= \sup_{\theb \in \Omega} \sum_{k=0}^N \frac{p(\theb)}{p(\theb_k)} |\ell^N_k(\theb)| \leq \sup_{\theb \in \Omega} \sum_{k=0}^N \frac{\|P'\|_\infty + \zeta}{\zeta} \frac{q_N(\theb)}{q_N(\theb_k)} |\ell^N_k(\theb)| \leq \left(1 + \frac{\| P' \|_\infty}{\zeta}\right) \Lambda^{q_N}_N, \\
		\Lambda^{q_N}_N &= \sup_{\theb \in \Omega} \sum_{k=0}^N \frac{q_N(\theb)}{q_N(\theb_k)} |\ell_k^N(\theb)| \leq \sup_{\theb \in \Omega} \sum_{k=0}^N \frac{\|P'\|_\infty + \zeta}{\zeta} \frac{p(\theb)}{p(\theb_k)} |\ell^N_k(\theb)| \leq \left(1 + \frac{\| P' \|_\infty}{\zeta}\right) \Lambda^p_N.
	\end{aligned}
\end{equation}
Furthermore, this expression can be used to demonstrate that results about the growth of the Lebesgue constant to a certain extent carry over to the setting of adaptively determined nodes. To see this, notice that if $\theb \in \Omega$ is given and $\theb_0, \dots, \theb_N$ are adaptively weighted Leja nodes, it holds for all $k = 0, \dots, N$ that
\begin{align}
	\zeta p(\theb) |\det V(\theb_0, \dots, \theb_{k-1}, \theb)| &\leq q_k(\theb) |\det V(\theb_0, \dots, \theb_{k-1}, \theb)| \\
	&\leq (\zeta + \| P' \|_\infty) p(\theb_k) |\det V(\theb_0, \dots, \theb_{k-1}, \theb_k)|.
\end{align}
Hence let $q(\theb)$ be as follows:
\begin{equation}
	q(\theb) = \begin{cases}
		\zeta + \| P' \|_\infty &\text{if $\theb = \theb_k$ for any $k = 0, \dots, N$}, \\
		\zeta &\text{otherwise}.
	\end{cases}
\end{equation}
Then it holds for all $k = 0, \dots, N$ that
\begin{equation}
	q(\theb) |\det V(\theb_0, \dots, \theb_{k-1}, \theb)| \leq q(\theb_k) |\det V(\theb_0, \dots, \theb_{k-1}, \theb_k)|.
\end{equation}
Hence there exists a single weighting function $q$ that defines these nodes. Moreover, following the same derivation as \eqref{eq:adaptivebound}, it holds that
\begin{equation}
	\label{eq:lebesguebound}
	\Lambda^p_N \leq \left(1 + \frac{\| P' \|_\infty}{\zeta} \right) \Lambda^q_N \text{ and } \Lambda^q_N \leq \left( 1 + \frac{\| P' \|_\infty}{\zeta} \right) \Lambda^p_N,
\end{equation}
where now $\Lambda^q_N$ is used instead of $\Lambda^{q_N}_N$ (i.e.\ the weighting function is independent from $N$).

Concluding, the Lebesgue constant of adaptively weighted Leja nodes grows asymptotically as least as slow as the Lebesgue constant of Leja nodes weighted with $q$. Furthermore, the growth of the Lebesgue constant weighted with the prior is similar to the growth of the Lebesgue constant weighted with $q$. As discussed before, it is difficult to assess these bounds theoretically, though the Lebesgue constant can often be assessed numerically. Notice that it is essential that $\zeta > 0$ for this result to hold, since otherwise the constant in \eqref{eq:lebesguebound} can become unbounded.

\subsection{Convergence of the posterior: the general case}
\label{subsec:convergence}
In this section we study convergence of the estimated posterior to the true posterior in the $\infty$-norm, given convergence of the interpolant. This demonstrates that, provided that the interpolant converges, a posterior constructed with the interpolant converges. 

As discussed previously, let $p(\theb)$, $p(\mathbf{z} \mid \theb)$, and $p(\theb \mid \mathbf{z})$ be the prior, likelihood, and posterior respectively. Let $u$ be the model, such that $u(\theb)$ is a model evaluation using a fixed set of parameters $\theb$. We assume an interpolant $u_N$ is given such that $\| u_N - u \|_{p(\theb)} \to 0$ for $N \to \infty$. Such an interpolant can for example be constructed with adaptively weighted Leja nodes, as discussed extensively in the previous section, but this is not explicitly assumed here (e.g.\ Leja nodes weighted with the prior also suit). Let $p_N(\mathbf{z} \mid \theb)$ and $p_N(\theb \mid \mathbf{z})$ be the likelihood and the posterior constructed with this interpolant, i.e.\ $p_N(\mathbf{z} \mid \theb) = P(u_N(\theb))$. The main result, stated in Theorem~\ref{thm:generalconvergence} below, is that if the interpolant converges to the model, the approximate posterior converges to the true posterior. We do not need differentiability of $P$ in this general case, but only require $P$ to be Lipschitz continuous.

\begin{theorem}
	\label{thm:generalconvergence}
	Let $u: \Omega \rightarrow \mathbb{R}$ be a continuous function and let $u_N$ be the interpolant of $u$ with $N$ nodes. Suppose
	\begin{equation}
		\| u_N - u \|_{p(\theb)} \leq C Q_N,
	\end{equation}
	with $Q_N \to 0$ for $N \to \infty$, and $C$ a positive constant (independent of $N$). Assume the likelihood (i.e.\ the function $P$) is Lipschitz continuous.

	Then
	\begin{equation}
		\| p_N(\theb \mid \mathbf{z}) - p(\theb \mid \mathbf{z}) \|_\infty \leq K Q_N,
	\end{equation}
	where $K$ is a positive constant.
\end{theorem}
\begin{proof}
	Recall the definition of $P$ from \eqref{eq:definitionofP}: $p(\mathbf{z} \mid \theb) = P(u(\theb))$. Let $D$ be the Lipschitz constant of $P$. Convergence readily follows:
	\begin{align*}
		\| p_N(\theb \mid \mathbf{z}) - p(\theb \mid \mathbf{z}) \|_\infty &= \| (p_N(\mathbf{z} \mid \theb) - p(\mathbf{z} \mid \theb)) p(\theb) \|_\infty \\
		&= \| [P(u_N(\theb)) - P(u(\theb))] p(\theb) \|_\infty \\
		&\leq D \| (u_N(\theb) - u(\theb)) p(\theb) \|_\infty \\
		&= D \| u_N(\theb) - u(\theb) \|_{p(\theb)} \\
		&\leq D C Q_N. \qedhere
	\end{align*}
\end{proof}
If $u_N$ converges to $u$ in the weighted $p(\theb)$-norm, the estimated posterior converges to the true posterior with at least the same rate of convergence, e.g.\ exponential if $Q_N \propto A^{-N}$ (for $A > 1$) and algebraic if $Q_N \propto N^{-\alpha}$ for $\alpha > 0$. This concludes the proof of Theorem~\ref{thm:generalconvergence}, extending previous work \cite{Marzouk2009,Birolleau2014} from Gaussian likelihoods to Lipschitz continuous likelihoods.

\subsection{Convergence of the posterior: Kullback--Leiber divergence}
\label{subsec:KLconvergence}
We assess the convergence properties of our algorithm using the Kullback--Leibler divergence, which is often used in a Bayesian setting to measure distance between distributions.

\subsubsection{Convergence of the Kullback--Leibler divergence}
Given two probability density functions $p(x)$ and $q(x)$ defined on a set $\Omega$, the Kullback--Leibler divergence is defined as follows:
\begin{equation}
	\DKL{p(x)}{q(x)} = \int_\Omega p(x) \log \frac{p(x)}{q(x)} \dd x.
\end{equation}
The Kullback--Leibler divergence is always positive, equals 0 if (and only if) $p \equiv q$, and is finite if $p(x) = 0$ implies $q(x) = 0$ (here, it is used that $\lim_{x \to 0} x \log x = 0$).

We are interested in proving bounds on $\DKL{p_N(\theb \mid \mathbf{z})}{p(\theb \mid \mathbf{z})}$ or $\DKL{p(\theb \mid \mathbf{z})}{p_N(\theb \mid \mathbf{z})}$. This is non-trivial due to the logarithm in the integral. To prove convergence, we first prove pointwise convergence of the logarithm, extend this to convergence of the integral using Fatou's lemma and finally conclude that convergence is attained. As the Kullback--Leibler divergence is defined for probability density functions, we have to incorporate the scaling of the posterior again. To this end, let $\gamma_N$ and $\gamma$ be defined as follows:
\begin{align}
	\label{eq:gamma}
	\gamma &= \int_\Omega p(\mathbf{z} \mid \theb) p(\theb) \dd \theb, \\
	\gamma_N &= \int_\Omega p_N(\mathbf{z} \mid \theb) p(\theb) \dd \theb.
\end{align}

To start off, the following lemma provides pointwise convergence of $\log \frac{p_N(x)}{p(x)}$. We omit the proof.
\begin{lemma}
	\label{lmm:pointwise}
	Let $g_n: \Omega \rightarrow \mathbb{R}$ be a series of functions with $g_n(x) \to g(x)$ for $n \to \infty$, for all $x \in \Omega$. Assume $g_n > 0$ for all $n$ and $g > 0$. Then
	\begin{equation}
		\log \frac{g_n(x)}{g(x)} \to 0, \text{ for $n \to \infty$, for all $x \in \Omega$}.
	\end{equation}
\end{lemma}

Note that the generalization to \emph{uniform} convergence is not trivial. By definition the Kullback--Leibler divergence does not require uniform convergence, but only convergence in the integral. As the functions we are using are probability density functions, Fatou's lemma is handy. It is well-known and we omit the proof.

\begin{lemma}[Fatou's lemma]
	\label{lmm:fatou1}
	Let $f_1, f_2, \dots$ be a sequence of extended real-valued measurable functions. Let $f = \limsup_{n \to \infty} f_n$. If there exists a non-negative integrable function $g$ (i.e.\ $g$ measurable and $\int_\Omega g \dd \mu < \infty$) such that $f_n \leq g$ for all $n$, then
	\begin{equation}
		\limsup_{n \to \infty} \int_\Omega f_n \dd \mu \leq \int_\Omega f \dd \mu.
	\end{equation}
\end{lemma}

We are now in a position to prove convergence of the Kullback--Leibler divergence, given pointwise convergence of the posterior. As uniform convergence of the posterior has been studied extensively in Section~\ref{subsec:convergence}, assuming pointwise convergence is not a restriction. However, we additionally assume positivity of the posterior, as the Kullback--Leibler becomes undefined otherwise.

\begin{theorem}
	\label{thm:klconvergence1}
	Suppose $\| p_N(\theb \mid \mathbf{z}) - p(\theb \mid \mathbf{z}) \|_\infty \to 0$ for $N \to \infty$, $p_N(\theb \mid \mathbf{z}) > \vareps \, p(\theb) > 0$, and $p(\theb \mid \mathbf{z}) > 0$ in $\Omega$. Then
	\begin{equation}
		\DKL{p(\theb \mid \mathbf{z})}{p_N(\theb \mid \mathbf{z})} \to 0.
	\end{equation}
\end{theorem}
\begin{proof}
	The proof consists of combining Lemma~\ref{lmm:pointwise}~and~\ref{lmm:fatou1}. The result follows from applying Lemma~\ref{lmm:fatou1} with $f_N = \log \frac{p}{p_N}$ and $f = 0$, in conjunction with $g = \sup_N \log \frac{p}{p_N}$. A direct application of this lemma yields $\DKL{p(\theb \mid \mathbf{z})}{p_N(\theb \mid \mathbf{z})} \to 0$. However, to apply Lemma~\ref{lmm:fatou1}, pointwise convergence of $\log \frac{p}{p_N}$ to 0 is necessary. This can easily be seen by applying Lemma~\ref{lmm:pointwise}, with $g_N = p_N$ and $g = p$.
\end{proof}

In a similar way, convergence of $\DKL{p_N(\theb \mid \mathbf{z})}{p(\theb \mid \mathbf{z})} \to 0$ can be proved.

\subsubsection{Convergence rate of the Kullback--Leibler divergence}
So far Theorem~\ref{thm:klconvergence1} only demonstrates convergence of the Kullback--Leibler divergence. The exact rate of convergence (or the constant involved in it) has not been deduced. In this section we will prove that the convergence rate doubles under general assumptions. These assumptions are:
\begingroup
\renewcommand\theenumi{(A.\arabic{enumi})}
\begin{enumerate}
	\item \label{ass:1} $\| p_N(\theb \mid \mathbf{z}) - p(\theb \mid \mathbf{z}) \|_\infty \leq C Q_N$, with $Q_N \to 0$ for $N \to \infty$,
	\item \label{ass:2} $p_N(\mathbf{z} \mid \theb) > 0$ and $p(\mathbf{z} \mid \theb) > 0$,
	\item \label{ass:3} $p(\theb)$ has compact support.
\end{enumerate}
\endgroup
The first assumption states that the estimated posterior converges, which can be shown using Theorem~\ref{thm:generalconvergence}. If $p_N$ converges algebraically to $p$ with rate $\alpha$, then $Q_N = N^{-\alpha}$. Many statistical models (such as the model from \eqref{eq:likelihood} with a uniform prior) fit in these assumptions. Note that the last assumption restricts the prior such that it has bounded support, so priors on an unbounded domain (e.g.\ the improper uniform prior or Jackson's prior) cannot be used in the setting of this section. The convergence result reads as follows.

\begin{theorem}
	\label{thm:klconvergence3}
	Suppose \ref{ass:1}, \ref{ass:2}, and \ref{ass:3} hold.
	
	Then
	\begin{equation}
		\DKL{p(\theb \mid \mathbf{z})}{p_N(\theb \mid \mathbf{z})} \leq K Q_N^2.
	\end{equation}
\end{theorem}
\begin{proof}
	The Kullback--Leibler divergence is always positive, hence
	\begin{align*}
		\DKL{p(\theb \mid \mathbf{z})}{p_N(\theb \mid \mathbf{z})} &\leq \DKL{p_N(\theb \mid \mathbf{z})}{p(\theb \mid \mathbf{z})} + \DKL{p(\theb \mid \mathbf{z})}{p_N(\theb \mid \mathbf{z})} \\
		&= \int_\Omega p_N(\theb \mid \mathbf{z}) \log \frac{p_N(\theb \mid \mathbf{z})}{p(\theb \mid \mathbf{z})} \dd \theb + \int_\Omega p(\theb \mid \mathbf{z}) \log \frac{p(\theb \mid \mathbf{z})}{p_N(\theb \mid \mathbf{z})} \dd \theb \\
		&= \int_\Omega (p_N(\theb \mid \mathbf{z}) - p(\theb \mid \mathbf{z})) \log \left(\frac{\gamma}{\gamma_N} \frac{p_N(\mathbf{z} \mid \theb)}{p(\mathbf{z} \mid \theb)}\right) \dd \theb, \\
	\intertext{with $\gamma$ and $\gamma_N$ according to \eqref{eq:gamma}. By taking the absolute value of the integral, we can bound it using the $\infty$-norm as follows.}
		\DKL{p(\theb \mid \mathbf{z})}{p_N(\theb \mid \mathbf{z})} &\leq \int_\Omega |p_N(\theb \mid \mathbf{z}) - p(\theb \mid \mathbf{z})| \left| \log \left(\frac{\gamma}{\gamma_N} \frac{p_N(\mathbf{z} \mid \theb)}{p(\mathbf{z} \mid \theb)}\right) \right| \dd \theb \\
		&\leq \left\| \log \left(\frac{\gamma}{\gamma_N} \frac{p_N(\mathbf{z} \mid \theb)}{p(\mathbf{z} \mid \theb)}\right) \right\|_\infty \int_\Omega |p_N(\theb \mid \mathbf{z}) - p(\theb \mid \mathbf{z})| \dd \theb. \\
	\intertext{Then, by working out the first formula we obtain:}
		\DKL{p(\theb \mid \mathbf{z})}{p_N(\theb \mid \mathbf{z})} &\leq \| \log(\gamma) - \log(\gamma_N) + \log(p_N(\mathbf{z} \mid \theb)) - \log(p(\mathbf{z} \mid \theb)) \|_\infty \\
		&\phantom{{}\leq}\cdot \int_\Omega |p_N(\mathbf{z} \mid \theb) - p(\mathbf{z} \mid \theb)| p(\theb) \dd \theb. \\
	\intertext{By application of the triangle inequality the first term can be bounded. By the boundedness of $\Omega$, the latter term can be bounded. Both yield $\infty$-norms as follows:}
		\DKL{p(\theb \mid \mathbf{z})}{p_N(\theb \mid \mathbf{z})} &\leq \left(\| \log(\gamma) - \log(\gamma_N) \|_\infty + \| \log(p_N(\mathbf{z} \mid \theb)) - \log(p(\mathbf{z} \mid \theb)) \|_\infty\right) \\
		&\phantom{{}\leq}\cdot \| p_N(\mathbf{z} \mid \theb) - p(\mathbf{z} \mid \theb) \|_\infty. \\
	\intertext{As $\gamma > 0$, the first term can be bounded using the Lipschitz continuity of the logarithm. Moreover, $\Omega$ is compact, hence there exists an $A > 0$ such that $p(\mathbf{z} \mid \theb) > A$ with $\theb \in \Omega$. Therefore the second term can also be bounded using the Lipschitz continuity of the logarithm. The last term is already in the right format, and we obtain}
		\DKL{p(\theb \mid \mathbf{z})}{p_N(\theb \mid \mathbf{z})} &\leq \frac{1}{|\log A|}\left(\| \gamma - \gamma_N \|_\infty + \| p_N(\mathbf{z} \mid \theb) - p(\mathbf{z} \mid \theb) \|_\infty\right) \| p_N(\mathbf{z} \mid \theb) - p(\mathbf{z} \mid \theb) \|_\infty. \\
	\intertext{Finally, the result readily follows:}
		\DKL{p(\theb \mid \mathbf{z})}{p_N(\theb \mid \mathbf{z})} &\leq \left(C_1 Q_N + C_2 Q_N\right) C_2 Q_N = K Q_N^2. \qedhere
	\end{align*}
\end{proof}
Note that in a similar manner it can be proved that $\DKL{p_N(\theb \mid \mathbf{z})}{p(\theb \mid \mathbf{z})} \leq K Q_N^2$. If $Q_N = N^{-\alpha}$, then $Q_N^2 = N^{-2 \alpha}$, demonstrating that the rate of convergence indeed doubles.

\section{Numerical results}
\label{sec:numerics}
Three numerical test cases are employed to show the performance of our methodology. First, in Section~\ref{subsec:explicit} two explicit test cases are studied, which are cases where an expression for $u$ is known that can be evaluated accurately such that the exact posterior can be determined explicitly. We use these cases to verify the theoretical properties that have been derived in Section~\ref{sec:convergence}. For sake of comparison, these cases are studied using interpolation based on Clenshaw--Curtis nodes, Leja nodes, and the proposed adaptively weighted Leja nodes.

In Section~\ref{subsec:burgers}, we study calibration of the one-dimensional Burgers' equation. As an explicit solution is not used here, we can show the practical purposes of the interpolation procedure to problems that are defined implicitly. We determine the interpolant using Clenshaw--Curtis nodes and weighted Leja nodes. The last test case consists of the calibration of the closure parameters of the Spalart--Allmaras turbulence model. Here, a single evaluation of the likelihood is computationally expensive as it requires the numerical solution of the Navier--Stokes equations. In this case straightforward methods (such as MCMC) become intractable and therefore we will only study weighted Leja nodes.

Note that the quantity of interest in each case is the posterior and not the model. Therefore we mainly study the convergence in the posterior and to a lesser extent the convergence in the model.

\subsection{Explicit test cases}
\label{subsec:explicit}
We consider two analytic functions to demonstrate the applicability of the approach. Both functions are analytic in their domain of definition, but one of the functions cannot be represented globally by means of a power series expansion (which is often challenging in interpolation problems). The first function, a Gaussian function, has a large radius of convergence, such that a single power series expansion can be used to globally approximate the function accurately. The second function, a multi-variate extension of Runge's function, yields a power series expansion with a small radius of convergence such that a single power series expansion cannot be used to globally approximate the function. Both functions are defined for any dimension $d$.

\subsubsection{Gaussian function}
A well-known class of analytic functions is formed by Gaussian functions. We will use the following function to represent the model:
\begin{equation}
	\label{pt10-eq}
	u_d: {[0,1]}^d \rightarrow \mathbb{R}, \text{ with } u_d(\theb) \coloneqq \exp\left(- \sum_{k=1}^d \left(\vartheta_k-\frac{1}{2}\right)^2 \right).
\end{equation}
This function is a composition of the exponential function and a polynomial, which are both globally analytic. Hence also this function is globally analytic and can therefore be approximated well using polynomials. Consequently, any nodal set can be used to interpolate this function---even an equidistant set---so we use this test case merely for a sanity check of the procedure and the theory.

Two statistical models are employed to demonstrate the independence of our procedure from the likelihood. The first is the statistical model discussed before, i.e.\ $z_k = u_d(\theb) + \vareps_k$, with $\varepsb \sim \mathcal{N}(\mathbf{0}, \Sigma)$, $\Sigma = \sigma^2 I$, and $\sigma = 1/10$. As discussed before, the likelihood equals
\begin{equation}
	p_\mathcal{N}(\mathbf{z} \mid \theb) \propto \exp\left[-\frac{1}{2} \trans{(\mathbf{z} - u_d(\theb))} \Sigma^{-1} (\mathbf{z} - u_d(\theb))\right],
\end{equation}
where $\mathbf{z}$ is the vector containing the data. A data vector of $n = 20$ elements is generated by sampling from a Gaussian distribution with mean $u_d(1/4)$ and standard deviation $\sigma$. The subscript $\mathcal{N}$ refers to multivariate normal. For the second model, we do not write an explicit relation between the data and the model, but only impose the following likelihood:
\begin{equation}
	\label{eq:doubled}
	p_\beta(\mathbf{z} \mid \theb) \propto \begin{cases}
		(1 - e)^2 (1 + e)^2 & \text{if $|e| < 1$} \\
		0 & \text{otherwise}
	\end{cases}, \text{ with } e = \frac{\overline{z} - u_d(\theb)}{\overline{z}} \text{ and } \overline{z} = \frac{1}{n} \sum_{k=1}^n z_k.
\end{equation}
We call this likelihood the Beta likelihood (denoted with $\beta$), which we use because it has different characteristics than the Gaussian likelihood. As the standard deviation is significantly larger in the second case, the posteriors differ considerably. Both likelihoods are continuously differentiable, so we expect similar accuracy when applying the proposed algorithm. In both cases the prior is assumed to be uniform on the domain $[0,1]^d$.

\begin{figure}
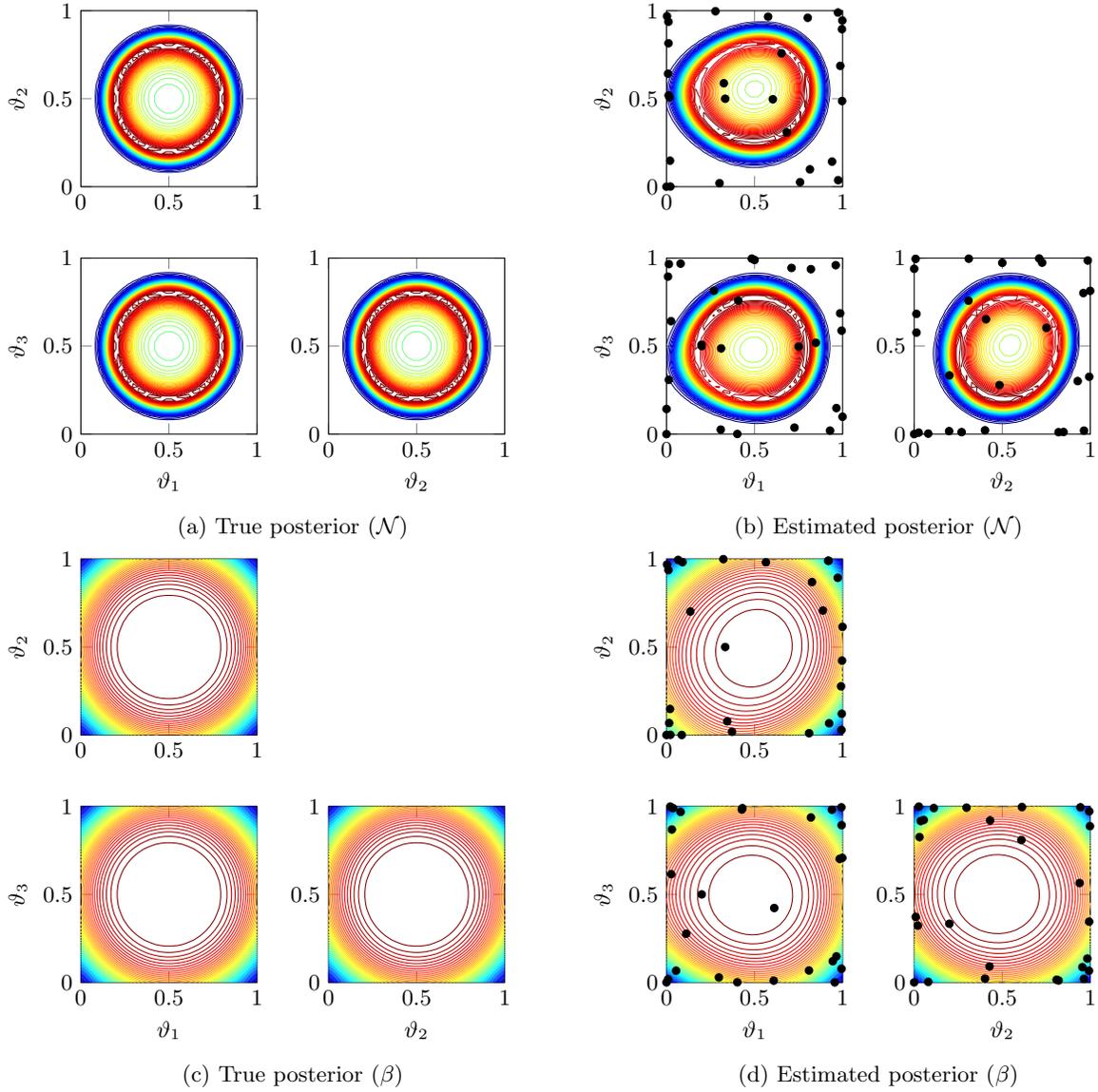

	\centering
	\footnotesize
	\begin{minipage}{.49\textwidth}
		\centering
		\includepgf{\textwidth}{\textwidth}{true_post.tikz}
		\subcaption{True posterior ($\mathcal{N}$)}
	\end{minipage}
	\begin{minipage}{.49\textwidth}
		\centering
		\includepgf{\textwidth}{\textwidth}{est_post_25.tikz}
		\subcaption{Estimated posterior ($\mathcal{N}$)}
	\end{minipage}

	\begin{minipage}{.49\textwidth}
		\centering
		\includepgf{\textwidth}{\textwidth}{true_post2.tikz}
		\subcaption{True posterior ($\beta$)}
	\end{minipage}
	\begin{minipage}{.49\textwidth}
		\centering
		\includepgf{\textwidth}{\textwidth}{est_post2_25.tikz}
		\subcaption{Estimated posterior ($\beta$)}
	\end{minipage}
	
	\caption{True and estimated posteriors with the analytic Gauss function as model using the multivariate normal ($\mathcal{N}$) likelihood and the Beta ($\beta$) likelihood with 25 Leja nodes. Red is high, blue is low.}
	\label{fig:analyticposteriors}
\end{figure}

Because the model under consideration is analytic, the value of $\zeta$ is not very important for the accuracy of the interpolation procedure (even $\zeta = 0$ works well in this case). We choose $\zeta=10^{-3}$, because then convergence of the posterior can be observed well, which is shown in Figure~\ref{fig:analyticposteriors} for $d = 3$. It is clearly visible that for the multivariate normal case, the nodes are placed more in the center of the domain (see for comparison Figure~\ref{fig:multileja}). This is also true for the second case, but less apparent due to the less intuitive structure of the posterior. The asymmetry between dimensions occurs due to the interpolation with Leja nodes, which are asymmetric by construction.

If we restrict ourselves to a one-dimensional case, the convergence of our algorithm can be assessed with high accuracy as it is possible to determine the Kullback--Leibler divergence with high accuracy using a quadrature rule. Moreover, a comparison with an interpolant based on Clenshaw--Curtis nodes can be performed. In higher dimensional cases such a comparison is not feasible, as determining the Kullback--Leibler divergence with high accuracy is intractable both with Monte Carlo (due to the relatively slow convergence) and with a quadrature rule (due to the deterioration of the high accuracy of the univariate quadrature rule). We want to emphasize here that assessing the convergence using the Kullback--Leibler divergence is essential, as the goal in this work is to construct an accurate \emph{posterior} (and not necessary an accurate \emph{interpolant}, for which various more efficient multivariate interpolation techniques exist).

The Kullback--Leibler divergence cannot be determined for the Beta likelihood. This is because two interpolants $u_{N_1}$ and $u_{N_2}$ in general do not have $|e| < 1$ exactly at the same $\theb$. Therefore the set where one model has $e = 0$ and the other $e > 0$ (and vice versa) is measurable, rendering the Kullback--Leibler divergence unbounded.

\begin{figure}
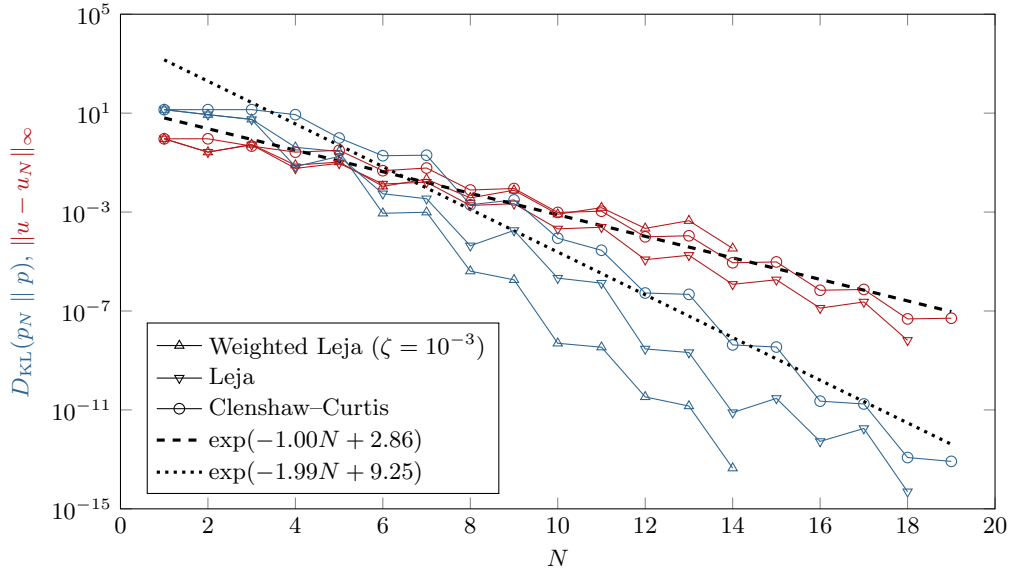

	\centering
	\small
	\includepgf{.8\textwidth}{.4944\textwidth}{analytic-convergence.tikz}
	\caption{Convergence for the Gaussian likelihood using the Gauss test function. The Kullback--Leibler divergence of the estimated posterior with respect to the true posterior is depicted in blue and the $\infty$-norm of the difference between the interpolant and the true model is depicted in red. The numerically estimated convergence rates of the Clenshaw--Curtis results are dashed and dotted.}
	\label{fig:analytic-convergence}
\end{figure}

The Kullback--Leibler divergence is determined using posteriors constructed with weighted Leja nodes, unweighted Leja nodes, and the Clenshaw--Curtis nodes (see Figure~\ref{fig:analytic-convergence}). All three nodal sets yield a model and a posterior that converge to respectively the true model and true posterior. The doubling of the convergence rate, according to Theorem~\ref{thm:klconvergence3}, already becomes apparent for the small number of nodes used and it is clearly visible that the weighted Leja nodes are mainly constructed for accuracy of the posterior instead of the model. The weighted Leja nodes show the fastest convergence, but the difference of model evaluations necessary to reach a certain accuracy level is small as the model under consideration is analytic.

\subsubsection{Runge's function}
A well-known test case for interpolation methods is the univariate Runge's function. We consider a multi-variate extension of it, defined (up to a constant) in the domain ${[0, 1]}^d$ as follows:
\begin{equation}
	u_d: [0,1]^d \rightarrow \mathbb{R}, \text{ with } u_d(\mathbf{x}) = \frac{5}{2+50 \sum_{k=1}^d {(x_k - \frac{1}{2})}^2}.
\end{equation}
This function is infinitely smooth, i.e.\ all derivatives exist and are continuous. However, its Taylor series has small radius of convergence. If a nodal set with exponentially growing Lebesgue constant is used to interpolate this function, the exponential convergence rate is significantly deteriorated (if it converges at all). This is well-known and typically called Runge's phenomenon. The Gaussian statistical model, the uniform prior, and the true value $\theb \equiv \frac{1}{4}$ are reconsidered for this test case. Results in terms of convergence using the three different nodal sets discussed previously are shown in Figure~\ref{fig:runge-convergence} (again $d=1$).

Differences between the nodal sets are more apparent than in the smooth test case. Again, the weighted Leja nodes are mainly constructed to obtain an accurate posterior and the accuracy of the interpolant is less important. It is clear that the weighted Leja nodes outperform the other nodal sets in the Kullback--Leibler norm. The unweighted Leja nodes and Clenshaw--Curtis nodes perform equally well, which is surprising as the Leja nodes do not have a logarithmically growing Lebesgue constant (contrary to the Clenshaw--Curtis nodes). Figure~\ref{fig:runge-convergence} also shows that upon increasing the value of $\zeta$ from $10^{-3}$ to $1$, the convergence rate of the weighted Leja sequence slightly decreases, as expected from theory. For small $N$, the posterior dominates the weighting function in Leja nodes, while for large $N$ the effect of $\zeta$ becomes important.

\begin{figure}
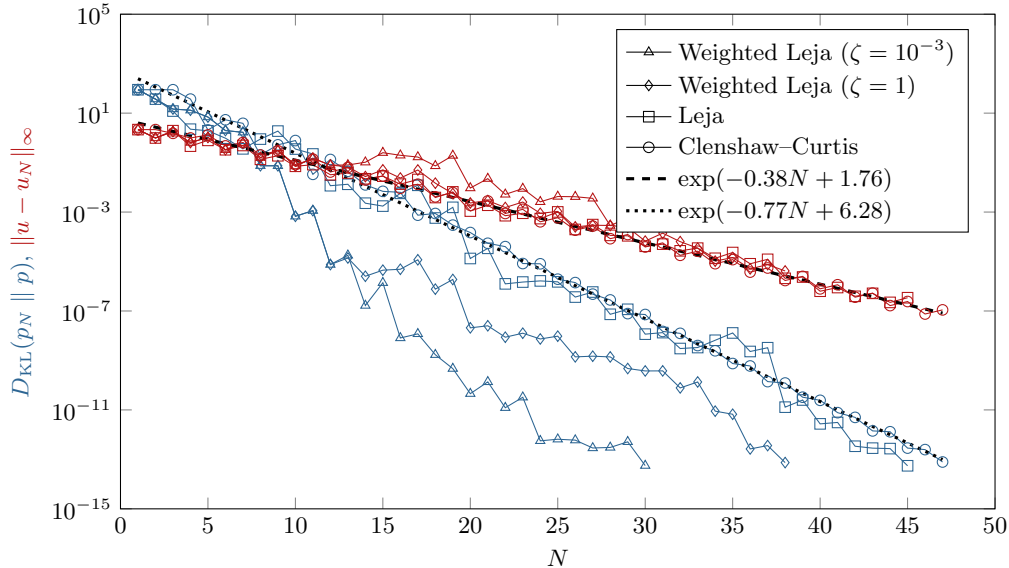

	\centering
	\small
	\includepgf{.8\textwidth}{.4944\textwidth}{runge-convergence.tikz}
	\caption{Convergence of the model (in red) and convergence of the Kullback--Leibler divergence (in blue) using four sampling algorithms for Runge's function. The numerically estimated convergence rates of the Clenshaw--Curtis results are dashed and dotted.}
	\label{fig:runge-convergence}
\end{figure}

\subsection{Burgers' equation}
\label{subsec:burgers}
In this section the Burgers' equation is considered where the boundary condition is unknown, extending the example from Marzouk~and~Xiu~\cite{Marzouk2009}. The one-dimensional steady state Burgers' equation for a solution field $y(x)$ and diffusion coefficient $\nu$ is stated as follows:
\begin{equation}
	0 = \nu \frac{\partial^2 y}{\partial x^2} - y \frac{\partial y}{\partial x},
\end{equation}
where boundary conditions complete the system. In this section, we will consider the equation on an interval $[-1, 1]$ with $\nu = 0.1$ and use the boundary conditions $y(-1) = 1 + \delta$ and $y(1) = -1$, where $\delta$ is unknown.

The inverse problem is as follows. Let $x_0$ be the zero-crossing of the solution $y$, i.e.\ $y(x_0) = 0$. Given a vector of noisy observations of $x_0$, we would like to obtain the value of $\delta$, i.e.\ reconstruct the boundary conditions from a specific value of the function. In the notation used so far, we would have a function $u: \mathbb{R} \rightarrow \mathbb{R}$ with $u(\delta) = x_0$. In the current setting, this function is only defined implicitly.

Let $\mathbf{z}$ be a vector with $n_d = 20$ noisy observations of $x_0$. We generate this vector by sampling $\delta$ from a normal distribution with mean $0.05$ and standard deviation $\sigma = 0.05$ and determining the corresponding values of $x_0$. We use a uniform prior such that $\delta \sim U[0, 0.1]$. The likelihood is Gaussian with zero mean and standard deviation $\sigma$, i.e.\ for each ``measurement'' $z_k$ (for $k = 1, \dots, n_d$) we have
\begin{equation}
	z_k - x_0(\delta) \sim \mathcal{N}(0, \sigma^2).
\end{equation}

A second-order finite volume scheme is employed to numerically solve the Burgers' equation equipped with a uniform mesh of $10^5$ cells. The zero-crossing is determined by piecewise interpolation of the solution.

\begin{figure}[t]
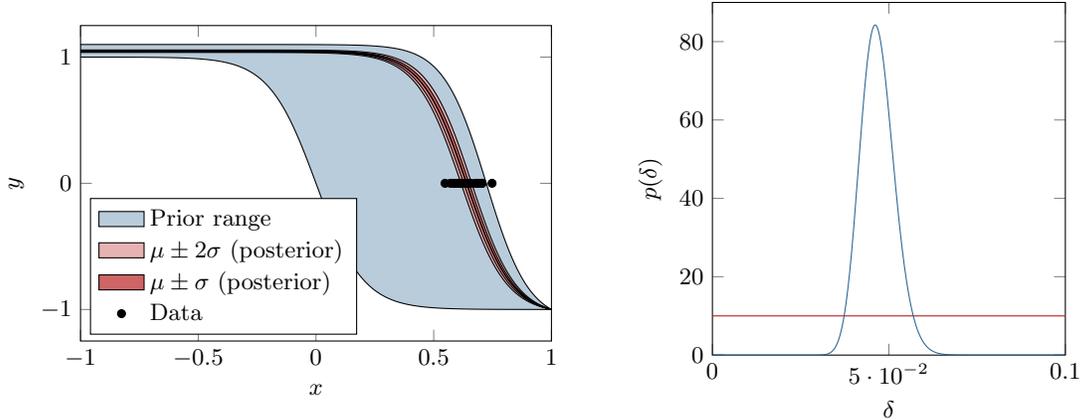

	\centering
	\begin{minipage}{.5\textwidth}
		\footnotesize
		\includepgf{.95\textwidth}{.7\textwidth}{burgers.tikz}
	\end{minipage}
	\begin{minipage}{.4\textwidth}
		\footnotesize
		\includepgf{.95\textwidth}{.95\textwidth}{burgers-posterior.tikz}
	\end{minipage}
	\caption{\emph{Left:} Propagation of the posterior through the Burgers' equations. \emph{Right:} The prior (in red) and the posterior (blue) of the Burgers' equation calibration test case.}
	\label{fig:burgers-mapost}
\end{figure}

We use Clenshaw--Curtis (scaled to ${[0,0.1]}$), Leja, and weighted Leja nodes to compute the posterior. In all three cases, the interpolant was refined until the Kullback--Leibler divergence of the estimated posterior with respect to the true posterior (computed using a fine quadrature rule) is smaller than $10^{-7}$. In Figure~\ref{fig:burgers-mapost} a sketch of the distribution of $u$ after propagation of this posterior and the exact posterior is depicted. The standard deviation of the output variable $u$ is small because we only take measurement errors into account (hence the standard deviation decreases for larger data sets). The posteriors determined after convergence of the three nodal sequences did not differ visually from the true posterior, so we do not present them separately.

There are large differences in the number of nodes that are necessary to obtain convergence, see Figure~\ref{fig:burgers-convergence}. For the posteriors determined using Clenshaw--Curtis nodes and Leja nodes, approximately 45 nodes are necessary. For the weighted Leja nodes, just 18 nodes are necessary for $\zeta=1$ and only 9 nodes are necessary for $\zeta = 10^{-3}$. This significant improvement is because the initial nodes already provide a good approximation of the posterior, which is leveraged when $\zeta$ is small. Again it is clearly visible that for small $N$, the posterior dominates the weighting function, while for large $N$ the value of $\zeta$ dominates. In both cases the weighted Leja nodes clearly outperform the other nodal sets.

\begin{figure}[t]
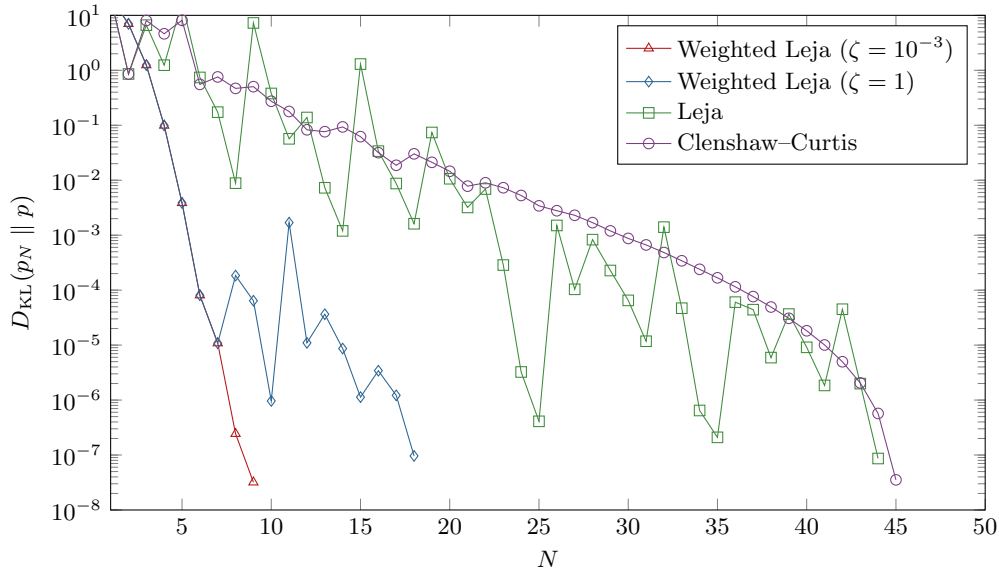

	\centering
	\includepgf{.8\textwidth}{.4944\textwidth}{burgers-convergence.tikz}
	\caption{Convergence of the posterior of Burgers' equation using three different sampling procedures.}
	\label{fig:burgers-convergence}
\end{figure}

\subsection{Turbulence model closure parameters}
\label{subsec:rae2822}
In this section we consider the flow around an airfoil (i.e.\ the two-dimensional cross section of a wing). The airfoil under consideration is the RAE2822, because extensive publicly available measurements have been performed for this particular airfoil. We use the wind tunnel measurements for the pressure coefficient from \citet{Cook1979} and study Case~6, i.e.\ the angle of attack equals $2.92^\circ$, the Mach number equals $0.725$, and the Reynolds number equals $6.5 \cdot 10^6$. These parameters are not corrected for wind tunnel effects, which is necessary for comparison with numerical simulations. See \citet{Slater2000} for more information about such a correction procedure and \citet{Cook1979} for more information about the measurement data under consideration, as these details are out of the scope of this article. The transonic flow around the airfoil in this case is depicted in Figure~\ref{fig:rae2822-cp}, determined numerically with the canonical turbulence coefficients of the Spalart--Allmaras turbulence model. It is clearly visible that there is shock formation on the upper side of the airfoil.

\begin{figure}
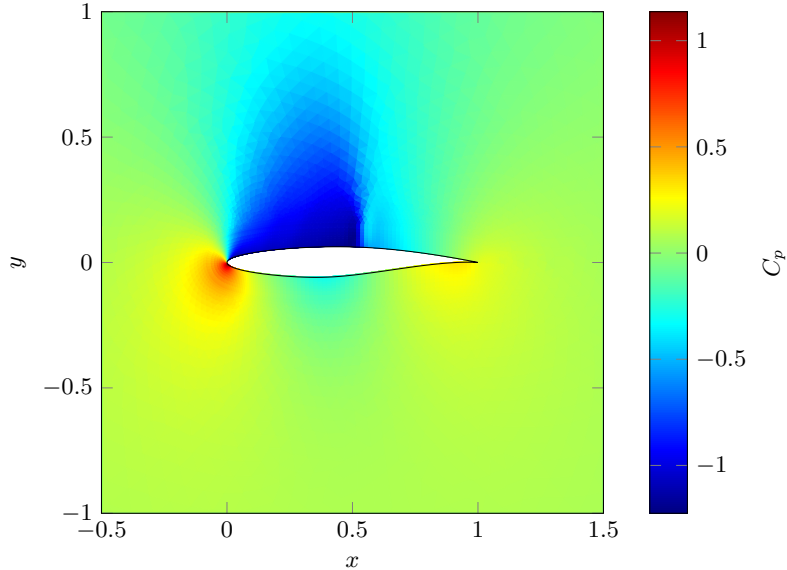

	\centering
	\includepgf{.5\textwidth}{.5\textwidth}{rae2822-cp.tikz}
	\caption{The pressure coefficient around the RAE2822 airfoil using the canonical turbulence coefficients.}
	\label{fig:rae2822-cp}
\end{figure}

\subsubsection{Problem description}
The flow around an airfoil is often modeled using the Navier--Stokes equations. Due to the large range of scales present in the solution of these equations, typically the Reynolds-averaged Navier--Stokes (RANS) equations are employed. The details are out of the scope of this article, we refer the interested reader to Wilcox~\cite{Wilcox1998}. The RANS equations do not form a closed set of equations and require a closure model. Typically, the Boussinesq hypothesis is used, which introduces an eddy viscosity. This viscosity models the effect of turbulence in the flow and requires an additional set of equations called a turbulence model. Throughout the years, many different turbulence models have been developed. These models have fitting parameters, chosen such that the model represents idealized test cases well. Often the coefficients are calibrated in a non-systematic way (e.g.\ by hand). A turbulence model commonly used for flow over airfoils is the Spalart--Allmaras (SA) turbulence model, which consists of a single equation modeling the transport, production, and dissipation of the eddy viscosity \cite{Spalart1992}.

Several forms of this model exist. The original SA turbulence model has 10 constants, typically defined as follows:
\begin{center}
	\begin{tabular}{r l | r l | r l}
		\hline
		$\sigma$ & $2/3$ & $C_{b1}$ & $0.1355$ & $C_{b2}$ & $0.622$ \\
		$\kappa$ & $0.41$ & $C_{w2}$ & $0.3$ & $C_{w3}$ & $2.0$ \\
		$C_{t3}$ & $1.2$ & $C_{t4}$ & $0.5$ & $C_{v1}$ & $7.1$  \\
		\hline
	\end{tabular}
\end{center}
Here we omit $C_{w1}$, which is commonly defined as $C_{w1} = C_{b1} / \kappa^2 + (1+C_{b2}) / \sigma$.

Although tested extensively, straightforward physical derivations do not exist for many of the parameter values above. In this section, we will calibrate these parameters systematically using Bayesian model calibration and apply the weighted Leja nodes proposed in this article. We will employ $\zeta = 1$, to make sure that we have convergence without spurious oscillations. This value is probably slightly too large to obtain a converging approximate posterior, but due to the large computational expense that is necessary for these simulations and the inability to run these simulations a large number of times we choose $\zeta$ rather too large than too small. Results on the calibration of turbulence parameters exist in literature (e.g.\ \cite{Edeling2014,Cheung2011}), where it is customary to calibrate the parameters using MCMC methods.

We will employ SU2 to numerically solve the RANS equations. SU2 is a second-order finite-volume computational fluid dynamics solver with support for the Spalart-Allmaras turbulence model. We have made a few minor modifications to allow the turbulence parameters to be changeable using configuration files. SU2 has been tested extensively to the airfoil test case (for the canonical turbulence parameters) with these options, see \cite{Palacios2014}.

\subsubsection{Statistical model}
We consider the same parameters for calibration as Edeling~et~al.~\cite{Edeling2014} and Cheung~et~al.~\cite{Cheung2011}, namely $\theb = \trans{(\kappa, \sigma, C_{b1}, C_{b2}, C_{v1}, C_{w2}, C_{w3} )}$. The remaining parameters are fixed at their canonical values.

We are now in the position to state the statistical model. Let $x$ be the spatial parameter along the contour of the airfoil, i.e.\ $x$ parametrizes the airfoil such that each $x$ has a unique value for the pressure coefficient. Suppose the data (i.e.\ a vector of pressure coefficients) $\mathbf{z} = \trans{(z_1, \dots, z_n)}$ is provided at locations $\mathbf{x} = \trans{(x_1, \dots, x_n)}$ on the airfoil. Then the statistical model under consideration is as follows:
\begin{equation}
\begin{aligned}
	v(\theb; x_k) &= u(\theb; x_k) + \delta(x_k), \\
	z(x_k) &= v(\theb; x_k) + \vareps_k,
\end{aligned}
\end{equation}
where $\delta(x) \sim \mathcal{N}(0, f(x, x'))$ is a Gaussian process with mean $0$ and squared exponential covariance function
\begin{equation}
	f(x, x') = A \exp\left[ -\frac{|x - x'|^2}{2 l^2} \right].
\end{equation}
Moreover, we have all $\vareps_k$ independently and identically normally distributed with mean zero and standard deviation $\tilde{\sigma}$, which is defined explicitly later (we use $\tilde{\sigma}$ to overcome confusion with one of the turbulence coefficients). The measurement data from \citet{Cook1979} has $n = 103$ measurement locations along the surface of the airfoil.

We choose the model such that $v$ is the ``true'' process that is modeled by $u$. $\delta$ and $\vareps_k$ are the model and measurement error respectively. Both require an estimation based on knowledge of the model and the data, as our calibration framework currently does not include hyperparameters. Bounds for the measurement error are provided in \citet{Cook1979}, resulting in $\tilde{\sigma} \approx 0.01$. This value is larger than the error of the measurement data, thus incorporating potential additional errors. The parameters of the Gaussian process are chosen to be $A = 0.6$ and $l = 0.2$, based on the largest error that we expect (which is approximately 0.5, based on the strength of the shock on top of the airfoil) and the distance between two measurement locations.

\subsubsection{Prior and likelihood}
The prior is assumed to be multivariate uniformly distributed with mean equal to the canonical values of the parameters and a variation of 50\%. The support of the prior is therefore as follows:
\begin{center}
	\begin{tabular}{r l}
		\hline
		$\kappa$ & $[0.205, 0.615]$ \\
		$\sigma$ & $[1/3, 1]$ \\
		$C_{b1}$ & $[0.0678, 0.2033]$ \\
		$C_{b2}$ & $[0.311, 0.933]$ \\
		$C_{v1}$ & $[3.55, 10.65]$ \\
		$C_{w2}$ & $[0.2, 0.4]$ \\
		$C_{w3}$ & $[1, 3]$ \\
		\hline
	\end{tabular}
\end{center}

The likelihood readily follows from the statistical model:
\begin{align}
	p(\mathbf{z} \mid \theb) &\propto \exp \left[ -\frac{1}{2} \trans{\mathbf{d}} C^{-1} \mathbf{d} \right], \\
	\intertext{with $\mathbf{d}$ the misfit and $C$ the covariance, i.e.}
	d_k &= z(x_k) - u(\theb; x_k), \\
	C &= \Sigma + C_f, \text{ with } C_f^{[i, j]} = f(x_i, x_j) \text{ and } \Sigma = \tilde{\sigma}I.
\end{align}
Finally, the posterior is formed by a multiplication of the likelihood with the prior in the usual way.

\subsubsection{Results}
\begin{figure}
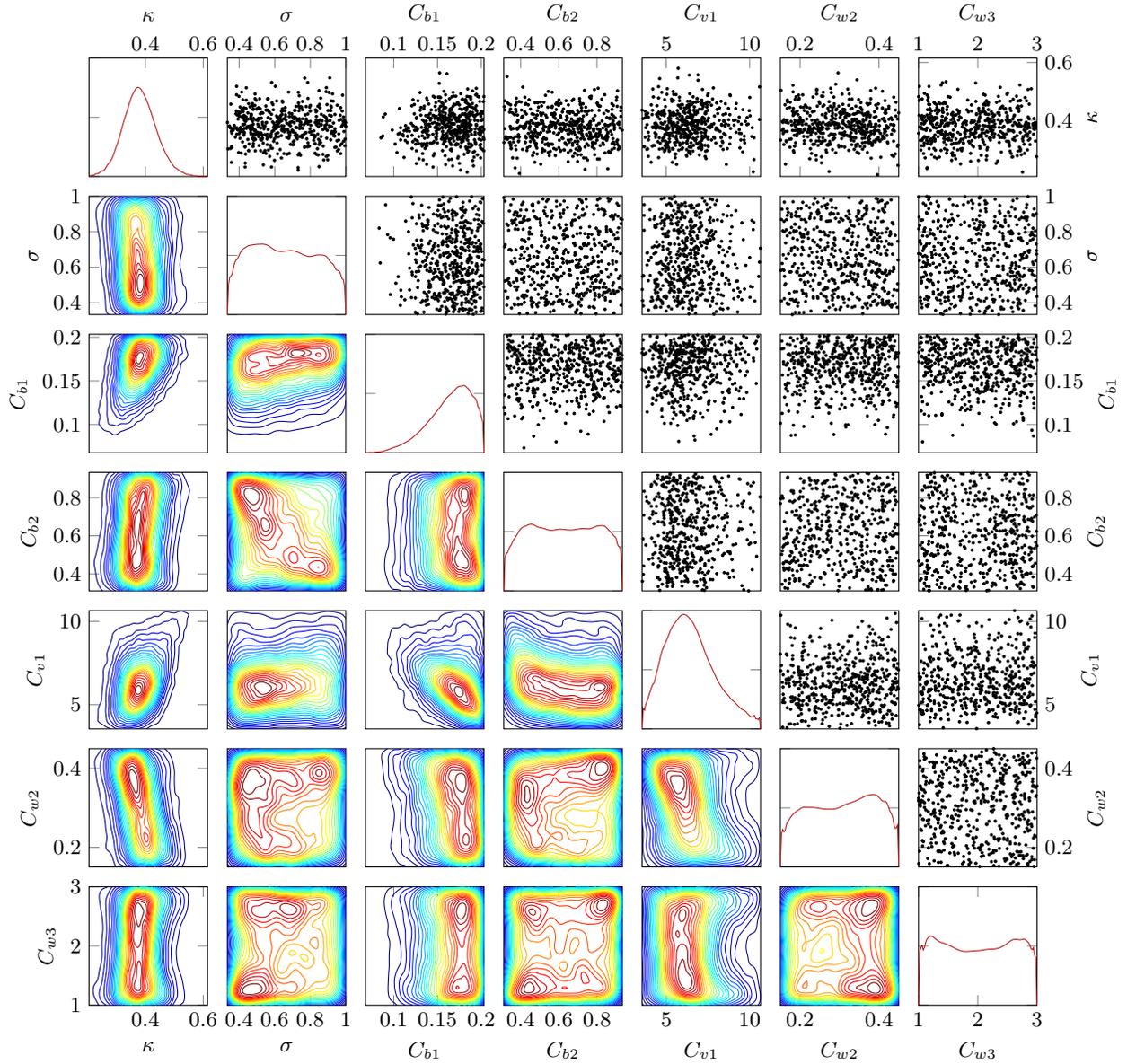

	\centering
	\footnotesize
	\includepgf{.85\textwidth}{.85\textwidth}{rae2822-posterior.tikz}
	\caption{The marginals of the posterior under consideration of the RAE2822 airfoil case. Red is high, blue is low.}
	\label{fig:megaposterior}
\end{figure}
We applied calibration with weighted Leja nodes to the current problem and used 100 nodes. The convergence was assessed using cross-validation between two consecutive interpolants, up to the point that visually no difference can be observed in the posterior. The one- and two-dimensional marginals from the obtained posterior are depicted in Figure~\ref{fig:megaposterior}. We cannot compare the result with the true posterior, as we did for the previous two test cases, so instead we compare the posterior with results obtained in literature.

The results are similar to \citet{Cheung2011}, who calibrated the same set of parameters using MCMC using more than 30\,000 samples. In their study, the parameters $\kappa$ and $C_{v1}$ are informed best, which is also clearly visible in Figure~\ref{fig:megaposterior}. In our results, the parameter $C_{b1}$ is also informed in comparison with the other parameters, but this is not the case in the study from \citeauthor{Cheung2011} There can be various reasons for this, among others the fact that \citeauthor{Cheung2011} used a flat plate test case for the calibration and used the skin friction coefficient as data.

In the study from \citet{Edeling2014} a similar test case is performed. The parameters $\kappa$ and $C_{v1}$ have the highest Sobol' indices and are therefore most influential on the model output (which is the pressure coefficient in this case). This is consistent with our study and the study from \citeauthor{Cheung2011} The results for $C_{b1}$ again differ, which we attribute to the fact that \citeauthor{Edeling2014} also used the skin friction coefficient and the flat plate test case.

To conclude, the proposed method is capable of constructing a good approximate posterior using a fraction of the 30\,000 model evaluations that were necessary in the MCMC case. The posterior compares very well with both studies. By applying MCMC to the constructed surrogate, we can propagate the posterior and make a prediction under uncertainty. To illustrate this, the prior and the posterior are propagated through the surrogate to obtain uncertainty bounds on the pressure coefficient. The results are depicted in Figure~\ref{fig:rae2822-cp-enveloppe}. It is clearly visible that the largest uncertainty originating from the posterior is near the shock on top of the airfoil. The uncertainty bounds obtained by propagation of the prior are relatively large, demonstrating that the calibration indeed did improve the accuracy of the prediction. Notice that there are oscillations visible in the range obtained by propagating the prior, but not in the result obtained by propagating the posterior. This is an indication that the surrogate indeed is refined with respect to the posterior.

\begin{figure}
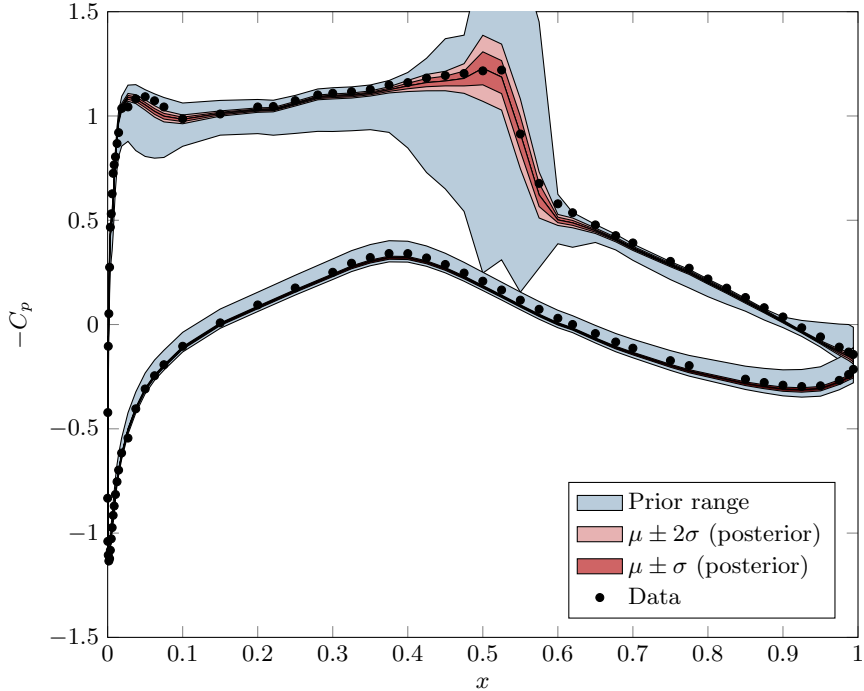

	\centering
	\includepgf{.7\textwidth}{.6\textwidth}{rae2822-meanstd.tikz}
	\caption{Pressure coefficient along the airfoil determined by propagation of the prior and the posterior through the interpolating surrogate. The standard deviation is shaded in red.}
	\label{fig:rae2822-cp-enveloppe}
\end{figure}

\section{Conclusion}
\label{sec:conclusion}
Bayesian model calibration is an attractive approach for calibrating model parameters of complex physical models. It is customary to sample the posterior using MCMC methods, but these are only tractable if the model under consideration can be evaluated rapidly. We consider problems where this is not the case, such as the calibration of turbulence closure coefficients.

Our proposed method consists of replacing the computationally expensive model with an interpolating surrogate model. The interpolant is determined adaptively using weighted Leja nodes, where the weighting function directly uses the posterior and is changed with each iteration. We have proposed a formulation in which the weighting function consists of two parts: a part incorporating the currently available posterior and a part incorporating the accuracy of the interpolant. The balance between these two can be changed adaptively by the parameter $\zeta$ and convergence is guaranteed for any $\zeta > 0$. The ``best'' value of $\zeta$ depends on the specifics of the model, the likelihood, and the posterior, and can be changed adaptively during the calibration procedure. Compared to conventional nodal sets (such as Clenshaw--Curtis nodes), we obtain more accurate results with less nodes.

Theoretically we have proved that if the interpolant converges to the true model in the $\infty$-norm, then the estimated posterior converges to the true posterior in the $\infty$-norm with a similar rate. Under mild conditions, the Kullback--Leibler divergence between the estimated and exact posterior converges with a doubled rate.

The three conducted numerical experiments confirm these theoretical findings: if the interpolant converges to the model, the Kullback--Leibler divergence converges to zero with doubled rate. For the explicit numerical test cases, this doubled convergence in the Kullback--Leibler divergence was clearly visible. The calibration of the Burgers' equation shows that the approach also works well for models that are defined implicitly.

Finally, calibration of turbulence closure parameters has been discussed, showing that the approach is truly applicable to computationally expensive models. We have compared the results from our calibration procedure with results conducted using Monte Carlo methods and the posteriors shows good resemblance, at a highly reduced computational cost.

\subsection{Future work}
Firstly, there are some limitations about the statistical model that can be used in our approach. We require that it can be written as a function of the model and no hyperparameters are used in the statistical model. Extending the current approach to such a setting is an important step, as hyperparameters are often employed in Bayesian model calibration.

Secondly, we believe the value of $\zeta$ should depend on the specifics of the model and further research is required to find the optimal value a priori or either dynamically during the procedure.

\section{Acknowledgments}
We thank Richard~Dwight for his valuable remarks on an early draft of this article. This research is part of the Dutch EUROS programme, which is supported by NWO domain Applied and Engineering Sciences and partly funded by the Dutch Ministry of Economic Affairs.

\bibliographystyle{plainnatnourl}
\bibliography{article-cited}

\end{document}